\documentclass[11pt,reqno,a4paper]{amsart}

\usepackage{etoolbox}
\usepackage{amsmath}
\usepackage{enumerate}
\usepackage{amssymb}
\usepackage{amscd}
\usepackage{amsthm}
\usepackage{amsfonts}
\usepackage{graphicx}
\usepackage[matrix, arrow, curve]{xy}
\usepackage{comment}
\usepackage{xcolor}
\usepackage[scaled]{helvet} 
\usepackage{courier} 
\usepackage[T1]{fontenc}                                                    
\usepackage{hyperref}

\normalfont

\numberwithin{equation}{section}

\newcommand{\SL}{\operatorname{SL}}

\newcommand{\SO}{\operatorname{SO}}

\newcommand{\cL}{\mathcal{L}}

\newcommand{\bN}{\mathbb{N}}

\newcommand{\bR}{\mathbb{R}}

\newcommand{\bT}{\mathbb{T}}

\newcommand{\bZ}{\mathbb{Z}}

\newcommand{\R}{\mathbb{R}}

\newcommand\subsetsim{\mathrel{%
\ooalign{\raise0.2ex\hbox{$\subset$}\cr\hidewidth\raise-0.8ex\hbox{\scalebox{0.9}{$\sim$}}\hidewidth\cr}}}
\newcommand{\eps}{\varepsilon}

\DeclareMathOperator{\supp}{supp}

\newcommand{\Q}{\mathbb Q}
\newcommand{\Z}{\mathbb Z}

\theoremstyle{theorem}
\newtheorem{theorem}{Theorem}[section]
\newtheorem{corollary}[theorem]{Corollary}
\newtheorem{proposition}[theorem]{Proposition}
\newtheorem{lemma}[theorem]{Lemma}

\newtheorem{problem}{Problem}

\theoremstyle{definition}
\newtheorem{definition}[theorem]{Definition}

\newtheorem{remark}[theorem]{Remark}
\newtheorem{example}[theorem]{Example}

\patchcmd{\subsection}{-.5em}{.5em}{}{}
\patchcmd{\subsubsection}{-.5em}{.5em}{}{}

\begin{document}

\title{Approximate lattices}

\author{Michael Bj\"orklund}
\address{Department of Mathematics, Chalmers, Gothenburg, Sweden}
\email{micbjo@chalmers.se}
\thanks{}

\author{Tobias Hartnick}
\address{Mathematics Department, Technion, Haifa 32000, Israel}
\curraddr{}
\email{hartnick@tx.technion.ac.il}
\thanks{}



\date{}

\dedicatory{}

\maketitle

\begin{abstract} 
In this article we introduce and study uniform and non-uniform approximate lattices in locally compact second countable (lcsc) groups. These
are approximate subgroups (in the sense of Tao) which simultaneously generalize lattices in lcsc group and mathematical quasi-crystals (a.k.a. Meyer sets) in lcsc abelian groups. 

We show that envelopes of strong approximate lattices are unimodular, and that approximate lattices in nilpotent groups are uniform. We also establish several results relating properties of approximate lattices and their envelopes. For example, we prove a version of the Milnor-Schwarz lemma for uniform approximate lattices in compactly-generated lcsc groups, which we then use to relate metric amenability of uniform approximate lattices to amenability of the envelope. 

Finally we extend a theorem of Kleiner and Leeb to show that the isometry groups of higher rank symmetric spaces of non-compact type are QI rigid with respect to finitely-generated approximate groups.
\end{abstract}

\section{Introduction}
\subsection{Approximate groups and approximate lattices}

In this article we introduce and study certain approximate subgroups of locally compact second countable (lcsc) groups which share many properties with lattices in such groups, and which we therefore propose to call ``approximate lattices''. 

The notion of an abelian ``approximate subgroup'' was already implicit in the early works by Freiman \cite{Freiman}, 
while the notion of a non-abelian ``approximate subgroup'' appears implicitly in the paper \cite{ErdSze} by 
Erd\"os and Szemeredi on the sum-product phenomenon, in the paper \cite{BouGam} by Bourgain and Gamburd 
on superstrong approximation and in the work \cite{Helfgott} by Helfgott on expansion in finite simple groups. 

The formal definition of an approximate subgroup (as recalled in Definition \ref{DefAG} below) was first put forward by Tao in \cite{Tao}. In this influential paper, the beginnings of the basic theory of such objects, based on previous fundamental works by Freiman \cite{Freiman}, Ruzsa \cite{Ruzsa} and Pl\"unnecke \cite{Pl}, were outlined. Since then, many groundbreaking results on \emph{finite} approximate subgroups have been established; in particular, Breuillard, Green and Tao established in \cite{BrGrTao2} their celebrated structure theorem for finite approximate groups. (We refer the reader to the surveys \cite{Breuillard} and \cite{BrSur} for more detailed bibliographies on these matters.) 

Developing a structure theory for general \emph{infinite} approximate groups, or even just general infinite groups is utterly hopeless. In geometric and measurable group theory one therefore often restricts the attention to infinite groups which admit interesting actions on metric, respectively measure, spaces. This leads to the study of lattices in lcsc groups. We recall that a subgroup $\Gamma$ of a lcsc group $G$ is called a \emph{lattice} if it is discrete and the homogeneous space $G/\Gamma$ admits a $G$-invariant probability measure. It is called a \emph{uniform lattice} if $G/\Gamma$ is moreover compact, and a \emph{non-uniform lattice} otherwise. 

The goal of this article is to extend these notions to the realm of approximate groups and to establish versions of some of the basic theorems concerning lattices in lcsc groups in this extended setting. Let us start by recalling the definition of an approximate (sub-)group:

\begin{definition}\label{DefAG}
A \emph{$k$-approximate group} is a pair $(\Lambda, \Lambda^\infty)$, where $ \Lambda^\infty$ is a group and $\Lambda \subset \Lambda^\infty$ is a subset such that 
\begin{enumerate}[({AG}1)]
\item $\Lambda$ is symmetric, i.e.\ $\Lambda = \Lambda^{-1}$, and contains the identity;
\item $\Lambda$ generates $ \Lambda^\infty$ as a group;
\item there exists a finite subset $F \subset  \Lambda^\infty$ of cardinality at most $k$ such that $\Lambda^2 \subset F\Lambda$. 
\end{enumerate}
If $(\Lambda, \Lambda^\infty)$ is a $k$-approximate group and $G$ is a group, then a homomorphism $\rho: \Lambda^\infty \to G$ of groups is called a \emph{representation} of $(\Lambda, \Lambda^\infty)$ and the image $\rho(\Lambda)$ of $\Lambda$ is called a \emph{$k$-approximate subgroup} of $G$. 
\end{definition}
By definition, a $1$-approximate subgroup is just a subgroup. Here we are mostly interested in approximate subgroups with $k \geq 2$. We emphasize that, unlike some authors, we do not assume our approximate subgroups to be finite.

Concerning the generalization of uniform lattices to the setting of approximate groups, we observe that a subgroup $\Gamma < G$ is a uniform lattice if and only if it is a Delone subset. Here a subset $\Lambda$ of a lcsc group $G$ is called a \emph{Delone set} if it is uniformly discrete and relatively dense with respect to some (hence any) proper left-invariant metric $d$ on $G$ which induces the given topology (see Subsection \ref{SecDelone} for a discussion of these notions). 
\begin{definition} Let $G$ be a lcsc group. An approximate subgroup $\Lambda \subset G$ is called a \emph{uniform approximate lattice} if it is a Delone set in $G$.
\end{definition}
The definition of a non-uniform approximate lattice is more involved, and we suggest two tentative definitions. Both definitions are based on the notion of the right-hull of a closed approximate subgroup $\Lambda \subset G$, which is a weak substitute for the homogeneous space $G/\Gamma$ of a closed subgroup. 

Given a lcsc group $G$ we denote by $\mathcal C(G)$ the compact space of closed subsets of $G$ with the Chabauty-Fell topology (see Subsection \ref{SecHull}). We then consider the action of $G$ on $\mathcal C(G)$ by $g.\Lambda := g\Lambda$, and given a closed subset $\Lambda \subset G$ define the \emph{right-hull} $X_\Lambda$ as the closure of the $G$-orbit of $\Lambda$ in $\mathcal C(G)$, i.e., 
\[
X_\Lambda := \overline{G.{\Lambda}}\; \subset\; \mathcal C(G).
\]
The right-hull of a closed subset is always compact by definition. If $\Lambda$ is not relatively dense in $G$, then it will contain the empty set.
\begin{definition} Let $G$ be a lcsc group. A uniformly discrete approximate subgroup $\Lambda \subset G$  is called a \emph{strong approximate lattice} if there exists a $G$-invariant probability measure $\nu$ on $X_\Lambda$ with $\nu(\{\emptyset\}) = 0$.
\end{definition}
Invariant measures on right-hulls of subsets of $\R^n$ have received much attention in the context of the dynamical approach to tilings and mathematical quasi-crystals (see \cite{Meyer}, and also \cite{BaakeGrimm} for a general overview and recent reference list of the subject). In the context of non-abelian groups, the study of invariant random subgroups and uniformly recurrent subgroups  \cite{IRS1, IRS2, URS} has led to an intensive study of invariant measures on conjugation hulls, i.e.\ orbit closures under the conjugation action of $G$. However, the right-hulls considered here are different from these conjugation hulls.

If $\Lambda$ is a uniform approximate lattice in an amenable group $G$, then it is also a strong approximate lattice. Indeed, by amenability of $G$ there exists an invariant probability measure on $X_\Lambda$, and one automatically has $\emptyset \not \in X_\Lambda$ in this case. For non-amenable groups, such an invariant measure need not exist, and we suggest the following weaker definition to deal with this case. Let us call a probability measure on $G$ \emph{admissible} if it is absolutely continuous with respect to the Haar measure class on $G$ and its support generates $G$ as a semigroup. Then on every compact $G$-space there is a probability measure $\nu$ which is \emph{$\mu$-stationary} in the sense that $\mu \ast \nu = \nu$ and we define:
\begin{definition} Let $G$ be a lcsc group. A uniformly discrete approximate subgroup $\Lambda \subset G$  is called an \emph{approximate lattice} if for every admissible probability measure $\mu$ on $G$ there exists a $\mu$-stationary probability measure $\nu$ on $X_\Lambda$ with $\nu(\{\emptyset\}) = 0$.
\end{definition}
With this definition, every uniform approximate lattice is an approximate lattice. It turns out that in amenable groups, every approximate lattice is strong, see Remark \ref{Strongness}.(1).
As for non-amenable groups, we do not currently know any example of an approximate lattice which is not strong. If one removes the symmetry condition on $\Lambda$, then such examples exist.

\begin{remark} Note that in the definition of a strong approximate lattices, we do not demand the invariant measure to be unique. In fact, there are many natural examples of strong approximate lattices (even in $\R$) which admit more than one non-trivial invariant measure on their hull, see e.g. Example \ref{Fish}. Similarly, in the case of approximate lattices, there may exist several non-trivial $\mu$-stationary probability measures on the hull for every given admissible probability measure $\mu$. Moreover, if $\mu$ and $\mu'$ are different admissible probability measures on $G$, then the $\mu$-stationary probability measures on the hull need not be related in any way to the $\mu'$-stationary probability measures.
\end{remark}

\subsection{Examples of approximate lattices}

Before we describe our results concerning approximate lattices, let us provide some important classes of examples of approximate lattices to convince the reader that the theory developed below has some content.

\begin{example} Every lattice in a lcsc group is a strong approximate lattice, and it is a uniform lattice if and only if is a uniform approximate lattice.
\end{example}
\begin{example}[see Corollary \ref{RelDenseSubsets}]\label{Ex1Intro} Every relatively dense and symmetric subset of a uniform approximate lattice containing the identity is again a uniform approximate lattice. In particular, relatively dense and symmetric subsets of uniform lattices containing the identity are uniform approximate lattices.
\end{example}

\begin{example}[see Proposition \ref{ModelSet}]\label{Ex2Intro} The following example goes back to Y. Meyer in the abelian case. Let $G$ and $H$ be lcsc groups and let $\Gamma < G \times H$ be a uniform lattice which projects injectively to $G$ and densely to $H$. We denote by $\pi_G: G\times H \to G$ the projection onto the first coordinate. Given a compact subset $W_0 \subset H$ with non-empty interior, the set
\[
\Lambda := \pi_G((G\times W_0) \cap \Gamma)
\]
is called a \emph{uniform model set} in $G$. If $W_0$ is moreover chosen symmetric (i.e.\ $W_0^{-1} = W_0$) and contains the identity, then $\Lambda \subset G$ is a uniform approximate lattice.
\end{example}
\begin{example} A relatively dense subset of a uniform model set is called a \emph{Meyer set}. Combining Example \ref{Ex1Intro} and Example \ref{Ex2Intro} we see that a symmetric Meyer set containing the identity is a uniform approximate lattice.
\end{example}
Meyer sets in $\R^n$ are among the most common models for mathematical quasi-crystals and have gained considerable interest (as witnessed by several hundred references in the bibliography of \cite{BaakeGrimm}). Meyer \cite{Meyer} proved (translated into our terminology) that every uniform approximate lattice in an abelian lcsc group is a Meyer set. We do not know whether this holds in more general classes of lcsc groups.
\begin{example}[see \cite{BHP}] Meyer's construction also applies to non-uniform lattices under some additional technical assumptions. Given a lattice $\Gamma < G \times H$ (not necessarily uniform) and a compact subset $W_0 \subset H$, the set
\[
\Lambda := \pi_G((G\times W_0) \cap \Gamma)
\] 
is called a \emph{regular model set} provided $W_0$ is Jordan-measurable with dense interior, aperiodic (i.e.\ ${\rm Stab}_H(W_0) = \{e\}$) and satisfies $\partial W_0 \cap \pi_H(\Gamma) = \emptyset$. If $\Lambda$ is a symmetric regular model set containing the identity, then it is a strong approximate lattice. This strong approximate lattice is uniform if and only if the underlying lattice is. This provides examples of strong approximate lattices, which are neither uniform nor contained in a lattice. In these examples, the invariant probability measure on $X_\Lambda \setminus\{\emptyset\}$ is actually unique (and in fact the unique $\mu$-stationary measure for every admissible $\mu$).
\end{example}

\subsection{Envelopes of approximate groups}

A lcsc group $G$ is called a \emph{(uniform) envelope} of an abstract group $\Gamma$ if $\Gamma$ is isomorphic to a (uniform) lattice in $G$. Similarly we call $G$ a (uniform) envelope of an abstract approximate group $(\Lambda, \Lambda^\infty)$ if there is an injective homomorphism $\rho: \Lambda^\infty \to G$ such that $\rho(\Lambda)$ is a (uniform) approximate lattice in $G$. The question of determining the lcsc groups which are (uniform) envelopes of (certain) groups has attained considerable attention recently (see e.g. \cite{BFS}), and we can ask the same question for approximate groups. The most basic necessary condition for a lcsc group $G$ to be an envelope of a group is unimodularity of $G$. This necessary condition carries over to the approximate setting in the following form.
\begin{theorem}[Unimodularity of envelopes, see Theorem \ref{ThmUnimodular}]\label{UnimodIntro} Every lcsc group which contains either a strong approximate lattice or a finitely-generated uniform approximate lattice is unimodular.
\end{theorem}
There remains the question which unimodular lcsc groups are envelopes. By a classical theorem of Borel--Harish-Chandra \cite{BoHC}, semisimple real Lie groups admits both uniform and non-uniform lattices. Using model sets one can show that they also contain both uniform and non-uniform approximate lattices, which are not contained in any lattice.

Remarkably, there also exist some $p$-adic semisimple group (like ${\rm SL}_n(\Q_p)$, see Remark \ref{padic}) which admit non-uniform approximate lattices, despite the fact that they do not admit non-uniform lattices. This is in contrast to the case of nilpotent Lie groups:
\begin{theorem}[Approximate lattices in nilpotent groups, see Theorem \ref{NilpotentUniform}]\label{NilpotentIntro} Every approximate lattice in a nilpotent lcsc group is uniform.
\end{theorem}
Not every nilpotent lcsc group admits a lattice. For example, by  a classical theorem of Malcev (see e.g. \cite[Thm. 2.12]{Raghunathan}) a simply-connected nilpotent Lie group admits a (uniform) lattice if and only if its Lie algebra admits a basis with rational structure constants. Using this criterion we provide in Corollary \ref{Bigger} an explicit example\footnote{The example was pointed out to us by Y. Benoist.} of a simply-connected nilpotent Lie group which admits uniform approximate lattices, but no lattices. In fact, in higher dimensions one can even find simply-connected nilpotent Lie groups which admit uniform approximate lattices, but are not quasi-isometric to any finitely-generated group (or even any vertex-transitive graph), see Example \ref{Elek}.


\subsection{Geometric properties of finitely-generated uniform approximate lattices}
If $G$ is a compactly generated lcsc group (for example, a finitely generated discrete group), then any two word metrics on $G$ with respect to compact generating sets are quasi-isometric. We refer to their common quasi-isometry (QI) class as the \emph{canonical QI class} of $G$ (see Subsection \ref{SecQIclass}). The study of the canonical QI class of finitely-generated groups is one of the main subjects of geometric group theory. Here we propose a generalization of this theory to approximate groups. We call an approximate group $(\Lambda, \Lambda^\infty)$ \emph{finitely generated} if $\Lambda^\infty$ is finitely generated as a group. In this case, all word metrics with respect to finite generating sets on $\Lambda^\infty$ restrict to quasi-isometric metrics on $\Lambda$, and we call their common QI class the \emph{canonical QI class} of $\Lambda$. With this terminology understood, the following theorem can be seen as a generalization of the Milnor-Schwarz lemma -- see Subsection \ref{SecMS} for an explanation how the present version translates into the more classical one in the case of groups.
\begin{theorem}[Milnor-Schwarz lemma for uniform approximate lattices, see Theorem \ref{MilnorSchwarz}]\label{MSIntro} Let $G$ be a compactly-generated lcsc group and $\Lambda \subset G$ a uniform approximate lattice. Then $\Lambda^\infty$ is finitely-generated, and the canonical QI class of $G$ restricts to the canonical QI class of $\Lambda$.
\end{theorem}
A sample application of Theorem \ref{MSIntro} (using also Theorem \ref{UnimodIntro}) is the following:
\begin{corollary}[Amenability, see Proposition \ref{PropAmenable}]\label{AmIntro} Let $G$ be a compactly generated lcsc group and $\Lambda \subset G$ a uniform approximate lattice. Then $\Lambda$ is metrically amenable with respect to its canonical QI class if and only if $G$ is an amenable lcsc group.
\end{corollary}
The classical QI rigidity problem concerns the converse of the Milnor--Schwarz lemma: Given a finitely-generated group $\Gamma$ quasi-isometric to $G$, is there a homomorphism $\rho: \Gamma \to G$ with finite kernel whose image is a uniform lattice in $G$? If $G$ is the isometry group of a higher rank symmetric space of non-compact type, then the answer to this question is positive by a celebrated result of Kleiner and Leeb \cite{KleinerLeeb} (also established independently, but slightly later by Eskin and Farb \cite{EskinFarb}). We extend their result to the context of finitely-generated approximate groups:
\begin{theorem}[QI rigidity with respect to approximate groups, see Theorem \ref{QIRigidity}]\label{QIIntro} Let $G$ be the isometry group of a higher rank symmetric space of non-compact type. If $(\Lambda, \Lambda^\infty)$ is a finitely-generated approximate group with $\Lambda$ quasi-isometric to $G$, then there exists a homomorphism $\rho: \Lambda^\infty \to G$ with small kernel such that $\rho(\Lambda)$ is a uniform approximate lattice in $G$.
\end{theorem}
The notion of ``small kernel'' appearing in the theorem is a (necessary) adaption of the finite kernel condition in the group case; see Subsection \ref{SecQI} for the precise definition.

This article is organized as follows: In Section \ref{SecUAL} we discuss the definition and basic examples of uniform approximate lattices. In Section \ref{SecAGGT} we introduce and discuss the canonical QI class of uniform approximate lattices. In analogy with the group case, we refer to this study as ``geometric approximate group theory''. In particular, we establish Theorem \ref{MSIntro} and Theorem \ref{QIIntro}. The discussion in Sections \ref{SecUAL} and \ref{SecAGGT} concerns only metric properties of lcsc groups and their Delone subsets, and is thus independent of the study of invariant and stationary measures on the hull. In Section \ref{SecNonUniform} we turn to general (i.e.\ not necessarily uniform) approximate lattices. We discuss different possible definitions of non-uniform approximate lattices and establish Theorem \ref{NilpotentIntro}. Section \ref{SecUnimodular} discusses the relation between measures on the hull of an approximate lattice and measures on the ambient lcsc group via periodization. This is then applied to prove Theorem \ref{UnimodIntro}. We also combine this theorem with Theorem \ref{MSIntro} to derive Corollary \ref{AmIntro}. For a more detailed list of subsections see the table of contents below.\\

{\bf Acknowledgments.} The authors thank Y. de Cornulier for numerous comments and clarifications and for pointing out Example \ref{ExCornulier}, and T. Dymarz for detailed explanations concerning the QI rigidity problem. They are indebted to Y. Benoist, A. Fish and E. Stark for suggesting the examples in Subsection \ref{SubsecLie}, Example \ref{Fish} and Example \ref{EStark} respectively. Finally, they thank B. Farb, A. Lubotzky and F. Pogorzelski for comments on an earlier draft and R. K\"ohl and A. Nevo for pointing out the references \cite{CapraceMonod} and \cite{CorwinGreenleaf} respectively.

\tableofcontents

\section{Uniform approximate lattices}\label{SecUAL}

\subsection{Delone sets in lcsc groups} \label{SecDelone}
Let $(X, d)$ be a metric space.  Given $R>0$ and $x \in X$ we denote by $B_R(x)$ and $\overline{B}_R(x)$ the open, respectively closed $d$-ball of radius $R$ around $x$, and given $A \subset G$ we denote by $N_R(A) := \bigcup_{x \in A} B_r(A)$ the $R$-neighbourhood of $A$. \begin{definition}\label{Delone}
Let $r, R >0$. A non-empty subset $\Lambda \subset X$ is called
\begin{enumerate}
\item $r$-\emph{uniformly discrete} if $d(x, y) \geq r$ for all $x,y \in \Lambda$;
\item $R$-\emph{relatively dense} if $N_R(\Lambda) = G$;
\item a $(r, R)$-\emph{Delone set} if it is both $r$-uniformly discrete and $R$-relatively dense.
\end{enumerate}
It is called \emph{uniformly discrete}, \emph{relatively dense} or a \emph{Delone set} if it has the respective property for some $r, R>0$.
\end{definition}
We will be interested in Delone sets in locally compact and second countable (lcsc) groups. By Struble's theorem  \cite[Thm. 2.B.4]{CdlH}, every lcsc group $G$ admits a proper left-invariant metric $d$ which induces the given topology on $G$. The following proposition shows that the notion of a Delone set in $(G, d)$ does not depend on the choice of $d$. In fact, it provides a purely topological characterization of Delone sets in groups. Here a subset $\Lambda \subset G$ is called \emph{left-syndetic} if  there exists a compact subset $K \subset G$ such that $\Lambda K  = G$. 
\begin{proposition}\label{TopCharDelone} Let $G$ be a lcsc group, $\Lambda \subset G$ a subset and $d$  a proper left-invariant metric on $G$ inducing the given topology.
\begin{enumerate}[(i)]
\item $\Lambda$ is uniformly discrete in $G$ if and only if the identity $e \in G$ is not an accumulation point of $\Lambda^{-1}\Lambda$.
\item $\Lambda$ is relatively dense in $G$ if and only if it is left-syndetic.
\end{enumerate}
In particular, the property of being a Delone set in $G$ is independent of the choice of $d$.
\end{proposition}
\begin{proof} (i) If $\Lambda$ is not uniformly discrete, then there exist elements $x_n \neq y_n \in \Lambda$ such that $d(x_n, y_n)< \frac 1 n$, hence $z_n := x_n^{-1}y_n \in  \Lambda^{-1}\Lambda$ satisfies $d(z_n, e) = d(y_n^{-1}x_n, e) = d(x_n, y_n) < \frac 1 n$, i.e.\ $\lim z_n = e$. Thus $e$ is not an isolated point of $\Lambda^{-1}\Lambda$. 

Conversely, if $e$ is not an isolated point of $\Lambda^{-1}\Lambda$ and $x_n, y_n \in \Lambda$ with $z_n := x_n^{-1}y_n \to e$, then
\[
d(x_n, y_n) = d(z_n, e)  \to 0,
\]
showing that $\Lambda$ is not uniformly discrete.

(ii) Assume that $\Lambda$ is $R$-relatively dense and let $K := \overline{B}_R(e)$. Given $g \in G$ there exists $x \in \Lambda$ such that $d(g, x) = d(x^{-1}g, e)  < R$, hence $x^{-1}g \in B_R(e)$ and thus $G = \Lambda K$.

Conversely, if $G = \Lambda K$ with $K$ compact, then every $g \in G$ can be written as $g = xk$ with $x \in \Lambda$, $k \in K$, whence
\[
d(g, x) = d(xk, x) = d(k, e) \leq \max_{k \in K} d(k, e) < \infty.\qedhere
\]
\end{proof}
It is well-known that every lcsc group contains a Delone set, see e.g. \cite[Prop. 3.D.11]{CdlH}. For our further discussion of Delone sets we will need a number of discreteness properties related to uniform discreteness. We say that $\Lambda$ is \emph{locally finite} if it is closed and discrete. Equivalently, the intersection with every compact set (equivalently, with every ball) is finite. We say that $\Lambda$ has \emph{(left) finite local complexity} (FLC) if $\Lambda^{-1}\Lambda$ is locally finite. We call $\Lambda$ \emph{left-uniformly locally finite} if for every compact set $K \subset G$ we have
\[
\sup_{g \in G} |\Lambda \cap gK| < \infty.
\] 
To relate these different notions of discreteness, we need the following alternative characterization of uniformly discrete sets.
\begin{lemma}\label{UniformlyDiscrete2} A subset $\Lambda \subset G$ is uniformly discrete if and only if there exists an open subset $V \subset G$ such that $|\Lambda \cap gV| \leq 1$ for all $g \in G$.
\end{lemma}
\begin{proof} If a set $V$ as in the lemma does not exist then for every $r>0$ we find $g_r, h_r \in \Lambda$ and $g \in G$ with $\{g_r, h_r\} \subset gB_r(e)$. Thus $g_r^{-1}h_r \in B_r(e)^2 \subset B_{2r}(e)$ and hence $g_r^{-1}h_r \to e$ as $r \to 0$. Then $e \in G$ is an accumulation point of $\Lambda^{-1}\Lambda$, and thus $\Lambda$ is not uniformly discrete. Conversely assume that a set $V$ as in the lemma exists. We may assume that $V$ contains $e$; then $\Lambda \cap \lambda V = \{\lambda\}$ for every $\lambda \in \Lambda$ and thus
\[
\Lambda^{-1}\Lambda \cap V = \bigcup_{\lambda \in \Lambda^{-1}} \lambda^{-1}(\Lambda \cap \lambda V) = \bigcup_{\lambda \in \Lambda^{-1}} \lambda^{-1}\{\lambda\} = \{e\},
\]
showing that $\Lambda$ is uniformly discrete.
\end{proof}

It follows from Lemma \ref{UniformlyDiscrete2} that every uniformly discrete set is left-uniformly locally finite, since every compact set can be covered by finitely many translates of $V$ as in the lemma. Combining this observation with Proposition \ref{TopCharDelone}.(i) we obtain the chain of implications
\begin{eqnarray}
&& \Lambda \text{ has FLC} \Rightarrow \Lambda^{-1}\Lambda \text{ discrete} \Rightarrow \Lambda \text{ uniformly discrete} \nonumber\\
&&  \Rightarrow \Lambda \text{ left-uniformly locally finite} \Rightarrow \Lambda \text{ locally finite} \Rightarrow \Lambda \text{ discrete}.\label{Discreteness}
\end{eqnarray}
\begin{remark} We have defined uniform discreteness in terms of left-invariant metrics, and consequently it is related to left-FLC and left-uniform local finiteness. We could define similar notions of right-uniform discreteness, right-FLC (local finiteness of $\Lambda\Lambda^{-1}$) and right-uniform local finiteness and obtain similar implications. However, in the sequel we are mostly interested in symmetric sets (i.e.\ $\Lambda$ satisfying $\Lambda = \Lambda^{-1}$) for which these distinctions do not matter, so we will not dwell on this point.
\end{remark}

\subsection{Definitions and first examples}

\begin{definition}
An \emph{approximate group} is a pair $(\Lambda, \Lambda^\infty)$, where $ \Lambda^\infty$ is a group and $\Lambda \subset \Lambda^\infty$ is a subset such that 
\begin{enumerate}[({AG}1)]
\item $\Lambda$ is symmetric, i.e.\ $\Lambda = \Lambda^{-1}$, and contains the identity;
\item $\Lambda$ generates $ \Lambda^\infty$ as a group;
\item there exists a finite subset $F \subset  \Lambda^\infty$ such that $\Lambda^2 \subset F\Lambda$. 
\end{enumerate}
A \emph{homomorphism} $\rho: (\Lambda, \Lambda^\infty) \to (\Xi, \Xi^\infty)$ of approximate groups is a group homomorphism $\rho: \Lambda^\infty \to \Xi^\infty$ which satisfies $\rho(\Lambda) \subset \Xi$.
\end{definition}
We consider groups as approximate groups by identifying a group $\Gamma$ with the approximate group $(\Gamma, \Gamma)$. We then refer to a homomorphism $(\Lambda, \Lambda^\infty) \to (\Gamma, \Gamma)$ as a \emph{representation} of $(\Lambda, \Lambda^\infty)$. The image of such a representation is then called an \emph{approximate subgroup} of $\Gamma$. This notion of approximate subgroup agrees with \cite{Tao, Breuillard}.

The definition of homomorphism of approximate groups defined here is stronger than the notion of a Freiman-homomorphism often used in approximate group theory. We will discuss this weaker notion briefly in Subsection \ref{SecFreiman} below.

Condition (AG3) is asymmetric. We could define a condition (AG3${}^{\rm op}$) by demanding that there is $F \subset \Lambda^\infty$ finite such that $\Lambda^2 \subset \Lambda F$. 
While (AG3) and (AG3${}^{\rm op}$) are not equivalent in general, they agree for symmetric sets $\Lambda$, since if $\Lambda = \Lambda^{-1}$ and $\Lambda^2 \subset F \Lambda$, then
\[\Lambda^2 = (\Lambda^2)^{-1} = (F\Lambda)^{-1} = \Lambda F^{-1}.\]
Replacing $F$ by $F \cup F^{-1}$ we may assume that $\Lambda^2 \subset F\Lambda \cap \Lambda F$ for every approximate group.

\begin{definition}
Let $G$ be a lcsc group. An approximate subgroup $\Lambda \subset G$ is called a \emph{uniform approximate lattice} if it is moreover a Delone subset. In this case $G$ is called an \emph{envelope} of $(\Lambda, \Lambda^\infty)$.
\end{definition}
Note that by definition $\Lambda \subset G$ is an approximate lattice if it is a Delone set which satisfies (AG1) and (AG3) for some $F \subset G$.
\begin{example} A subgroup $\Gamma < G$ is an approximate uniform lattice if and only if it is a uniform lattice. Indeed, if $\Gamma$ is Delone, then it is discrete by \eqref{Discreteness} and left-syndetic, hence a uniform lattice. Conversely, if $\Gamma$ is a uniform lattice, then it is left-syndetic, and  $\Gamma^{-1}\Gamma = \Gamma$ is discrete, hence $\Gamma$ is uniformly discrete by \eqref{Discreteness}.
\end{example}
In the last example, $\Lambda^\infty = \Gamma$ was also a discrete subset of $G$. However, discreteness of $\Lambda^\infty$ is not part of the definition of an approximate lattice, and we will discuss examples of approximate lattices $\Lambda \subset G$ in Subsection \ref{SecModelMeyer} below for which $\Lambda^\infty$ is actually dense in $G$. 
\begin{example} Relatively dense symmetric subsets of uniform lattices, and in fact of approximate uniform lattices are approximate uniform lattices provided the contain the identity. \end{example}
The proof is based on the following characterization of approximate lattices, which is of independent interest.
\begin{proposition}\label{PropDiscreteEquiv} Let $G$ be a lcsc group and assume $\Lambda = \Lambda^{-1} \subset G$ is relatively dense and contains the identity. Then the following are equivalent.
\begin{enumerate}[(i)]
\item $\Lambda$ is an approximate lattice.
\item $\Lambda^k$ is uniformly discrete for all $k \geq 1$.
\item $\Lambda^6$ is discrete.
\item $\Lambda^3$ is locally finite.
\end{enumerate}
\end{proposition}
\begin{proof}  (i) $\Rightarrow$ (ii) For all $k \geq 1$ we have $\Lambda^k \subset F\Lambda^{k-1}$ and hence $\Lambda^k \subset F^{k-1}\Lambda$ by induction. Since $F^{k-1}$ is finite and $\Lambda$ is discrete, $F^{k-1}\Lambda$ is discrete. Thus for all $k \geq 1$ we have that $\Lambda^{2k}$ is discrete, whence $\Lambda^k$ is uniformly discrete by \eqref{Discreteness}.

(iv) $\Rightarrow$ (i) If $\Lambda^3$ is locally finite, then $\Lambda$ is uniformly discrete by \eqref{Discreteness}. To establish (AG3), we define $F:= B_R(e) \cap \Lambda^3$, where $R$ is chosen in such a way that $\Lambda$ is $R$-relatively dense. Since $\Lambda^3$ is locally finite, $F$ is finite. Now let $z \in \Lambda^2$ and write $z = xy$ with $x,y \in \Lambda$. By definition of $R$ there exists $w \in \Lambda$ with $d_G(w, z^{-1}) < R$, hence $zw = zxy \in F$. Thus $z \in (zw)w^{-1} \in F\Lambda$.

Now the implication (ii) $\Rightarrow$ (iii) is obvious, and (iii) $\Rightarrow$ (iv) follows from \eqref{Discreteness}. This finishes the proof of the proposition.
\end{proof}
\begin{corollary}\label{RelDenseSubsets} Let $\Lambda \subset G$ be an approximate lattice and $\Lambda_0 \subset \Lambda$ be a symmetric subset containing the identity. Then the following are equivalent.
\begin{enumerate}[(i)] 
\item$\Lambda_0$ is an approximate lattice in $G$.
\item $\Lambda_0$ is relatively dense in $G$.
\item $\Lambda_0$ is relatively dense in $\Lambda$.
\end{enumerate}
\end{corollary}
\begin{proof} The equivalence (i)$\Leftrightarrow$(ii) is immediate from Proposition \ref{PropDiscreteEquiv} since a subset of a discrete set is discrete. If $\Lambda_0$ is relatively dense in $\Lambda$, then there exists a finite subset $K_1 \subset \Lambda$ and a compact subset $K_2 \subset G$ such that $\Lambda = \Lambda_0K_1$ and $G = \Lambda K_2$ and hence $G = \Lambda_0 K_1K_2$ showing that $\Lambda_0$ is relatively dense in $G$. Conversely, if $\Lambda_0$ is relatively dense in $G$, say $G = \Lambda_0 K$ with $K$ compact, then every $g \in \Lambda$ can be written as $g = \gamma k$ with $\gamma \in \Lambda_0$ and $k \in K$. Then $k = \gamma^{-1}g \in \Lambda^2$, hence $\Gamma = \Gamma_0 (K \cap \Lambda^2)$. Now $(K \cap \Lambda^2)$ is finite by Proposition \ref{PropDiscreteEquiv}, and hence $\Lambda_0$ is relatively dense in $\Lambda$.
\end{proof}
Among the various characterizations of approximate lattice in Proposition \ref{PropDiscreteEquiv}, (iii) is often the most convenient one, since one does not have to check that $\Lambda^6$ is closed. Note that if $\Lambda \subset G$ is a uniform approximate lattice, then by Proposition \ref{PropDiscreteEquiv}.(ii) the set $\Lambda^{-1}\Lambda = \Lambda^2$ is uniformly discrete, hence $\Lambda$ has FLC by \eqref{Discreteness}.

\subsection{Model sets and Meyer sets}\label{SecModelMeyer}
We now turn to examples of uniform approximate lattices which are not contained in uniform lattices. These examples are based on Meyer's construction of cut-and-project sets, which we recall briefly.
\begin{definition} A \emph{cut-and-project-scheme} is a triple $(G, H, \Gamma)$ where $G$ and $H$ are lcsc groups and $\Gamma < G \times H$ is a lattice which projects injectively to $G$ and densely to $H$. A cut-and-project scheme is called \emph{uniform} if $\Gamma$ is moreover a uniform lattice.
\end{definition}
Given a cut-and-project scheme $(G, H, \Gamma)$ we denote by $\pi_G$, $\pi_H$ the coordinate projections of $G \times H$ and set $\Gamma_G := \pi_G(\Gamma)$ and $\Gamma_H := \pi_H(\Gamma)$. We then define a map $\tau: \Gamma_G \to H$ as $\tau := \pi_H \circ (\pi_G|_\Gamma)^{-1}$. Note that the image of $\tau$ is precisely $\Gamma_H$; in the abelian case this map is sometimes called the ``$*$-map''.

\begin{definition} Let  $(G, H, \Gamma)$ be a cut-and-project scheme with associated ``$*$-map'' $\tau: \Gamma_G \to H$. Given a compact subset $W_0 \subset H$, the pre-image
\[
P_0(\Gamma, W_0) := \tau^{-1}(W_0) \subset G
\]
is called a \emph{weak model set}, and $W_0$ is called its \emph{window}. A weak model set is called a \emph{model set} if its window has non-empty interior. It is called \emph{regular} if the window $W_0$ is Jordan-measurable with dense interior, aperiodic (i.e.\ ${\rm Stab}_H(W_0) = \{e\}$) and satisfies $\partial W_0 \cap \Gamma_H = \emptyset$. A model set is called a \emph{uniform model set} if the underlying cut-and-project scheme is uniform. A \emph{Meyer set} is a relatively dense subset of a uniform model set.
\end{definition}
The following proposition goes back to Meyer \cite{Meyer} for abelian groups.
\begin{proposition}\label{ModelSet} $\Lambda = P_0(\Gamma, W_0)$ be a 
model set over $(G, H, \Gamma)$.
\begin{enumerate}[(i)]
\item $\Lambda$, and in fact $\Lambda^{-1}\Lambda$, is uniformly discrete. In particular, $\Lambda$ has finite local complexity.
\item $\Lambda$ satisfies (AG3). In particular, it is an approximate subgroup if its window is symmetric and contains the identity.
\item If $\Gamma$ is uniform, then $\Lambda$ is moreover relatively dense, hence a Delone set.
\item If $\Gamma$ is non-uniform, then $\Lambda$ is not relatively dense.
\end{enumerate}
In particular, a uniform model set is a uniform approximate lattice if its window is symmetric and contains the identity.
\end{proposition}
\begin{proof} (i) If $K \subset G$ is a compact subset, then
\[
\Lambda^{-1}\Lambda \cap K = \tau^{-1}(W_0)^{-1}\tau^{-1}(W_0) \cap K \subset \tau^{-1}(W_0^{-1}W_0) \cap K = \pi_G((K \times W_0^{-1}W_0) \cap \Gamma),
\]
which is finite since $\Gamma$ is locally finite and $K \times W_0^{-1}W_0$ is compact. It thus follows from \eqref{Discreteness} that $\Lambda$ is uniformly discrete, and the same argument applies to $\Lambda^{-1}\Lambda$.

(ii) Since $\Gamma_H = \tau(\Gamma_G)$ is dense in $H$ and $W_0$ has non-empty interior we have $\tau(\Gamma_G)W^o_0 = H$, and in particular $W_0^2$ is a subset of $\tau(\Gamma_G)W_0^o$. Since it is compact, there is actually a finite subset $F \subset \Gamma_G$ with $W_0^2 \subset \tau(F)W_0^o \subset \tau(F)W_0$. Then (AG3) follows from
\[
\Lambda^2 = \tau^{-1}(W_0)^2 \subset \tau^{-1}( \tau(F)W_0) \subset F \tau^{-1}(W_0) = F \Lambda.
\]
(iii) Assume now that $\Gamma < G\times H$ is cocompact. We then find compact subsets $K \subset G$ and $L \subset H$ such that $G \times H = (K \times L) \Gamma$. Since $L$ is compact we can argue as in the proof of (AG3) to find a finite subset $F \subset \Gamma_G$ such that $L^{-1} \subset \tau(F) (W_0^o)^{-1} \subset \tau(F)W_0^{-1}$. We deduce that
\begin{equation}\label{FLambdaInverse}
F\Lambda^{-1} = F \pi_G((G \times W_0^{-1}) \cap \Gamma) = \pi_G((G \times \tau(F)W_0^{-1}) \cap \Gamma) \supset \pi((G \times L^{-1}) \cap \Gamma).
\end{equation}
On the other hand, if $g \in G$, then $(g,e) \in G \times H = (K \times L) \Gamma$ and thus $ (K^{-1}g \times L^{-1}) \cap \Gamma \neq \emptyset$. Since $\pi_G|_\Gamma$ is injective we deduce with \eqref{FLambdaInverse} that
\[
\emptyset \neq   \pi_G((K^{-1}g \times L^{-1}) \cap \Gamma) \subset K^{-1}g \cap    \pi_G((G \times L^{-1}) \cap \Gamma) \subset K^{-1}g \cap F\Lambda^{-1},
\]
hence $g \in KF\Lambda^{-1}$. Since $g \in G$ was arbitrary we obtain $G = KF\Lambda^{-1}$ and hence $G = G^{-1} = \Lambda F^{-1}K^{-1}$, showing that $\Lambda$ is relatively dense.

(iv) Assume that $\Lambda$ is relatively dense and let $K \subset G$ compact with $G = \Lambda K$. Let $L$ be a compact identity neighbourhood in $H$. We claim that $G \times H = \Gamma(K \times W_0^{-1}U)$. Indeed, let $(g,h) \in G \times H$; since $\Gamma_H$ is dense in $H$ we find $\gamma_1 \in \Gamma_G$ and $u \in U$ such that $h = \tau(\gamma_1) u$. Since $G = \Lambda K$ we then find $\gamma_2 \in \Lambda = \tau^{-1}(W_0)$ and $k \in K$ with $\gamma_1^{-1}g = \gamma_2k$. Then
\[
(g, h) = (\gamma_1, \tau(\gamma_1))(\gamma_1^{-1}g, u) =  (\gamma_1, \tau(\gamma_1))(\gamma_2, \tau(\gamma_2))(k, \tau(\gamma_2)^{-1}u) \in \Gamma(K \times W_0^{-1}U).
\]
Since $(g,h) \in G \times H$ was arbitrary, this finishes the proof.
\end{proof}
\begin{remark} The assumption that the window $W_0$ has non-empty interior is crucial in the proof of Proposition \ref{ModelSet}, and the proposition fails for weak model sets. For example, let $V := \{(x,y)\in \Z^2 \mid {\rm gcd}(x,y) = 1\}$ be the set of visible lattice points in $\R^2$. Then $V$ is a weak model set, which is uniformly discrete and satisfies $V^2= \Z^2$, but is not relatively dense in $\Z^2$ (or equivalently $\R^2$, see e.g. \cite[Prop. 10.4]{BaakeGrimm}). Moreover, since $V^2 = \Z^2$ and $V$ is not relatively dense in $\Z^2$, it does not satisfy (AG3). Thus $V \cup \{(0,0)\}$ is not even an approximate group.
\end{remark}
Combining Proposition \ref{ModelSet} with Corollary \ref{RelDenseSubsets} we deduce:
\begin{example} A Meyer set is a uniform approximate lattice if it is symmetric and contains the identity.
\end{example}
The following converse theorem is due to Meyer \cite{Meyer}; see \cite{Moody} for a detailed discussion.
\begin{theorem}[Meyer] Every uniform approximate lattice in an abelian lcsc group is a Meyer set.
\end{theorem}
Currently we do not know a single example of a uniform approximate lattice in a non-abelian lcsc group for which we can show that it is not Meyer.
\begin{problem} For which lcsc groups $G$ is every uniform approximate lattice in $G$ a Meyer set?
\end{problem}
\begin{remark} A related question is whether certain classes of approximate lattices are contained in specific classes of model sets. For example, it is proved in \cite{BF} that every ``sufficiently aperiodic'' $2$-approximate lattice in a countable amenable group is a subsets of a Sturmian set (a very special model set) with the same upper Banach density.
\end{remark} 

\subsection{Uniform approximate lattices in nilpotent Lie groups without lattices}\label{SubsecLie}

The goal of this subsection is to show that there exist Lie groups which admit approximate uniform lattices, but no lattices. In fact, these Lie groups can be chosen to be nilpotent and of dimension $7$. The following concrete example below was pointed out to us by Y. Benoist.

The list of all $7$-dimensional nilpotent real Lie algebras is provided in \cite{Gong}. According to item (123457I) of this list, there exists a family of pairwise non-isomorphic nilpotent real Lie algebra $\mathfrak g_\lambda$ parametrized by $\lambda \in \R \setminus \{0,1\}$ with basis $X_1, \dots, X_7$ and bracket relations
\[
[X_1, X_i] = X_{i+1} \; (2 \leq i \leq 6), \quad [X_2, X_3] = X_5, \quad [X_2, X_4] = X_6,\]
\[[X_2, X_5] = \lambda X_7, \quad [X_3, X_4] = (1-\lambda)X_7,\]
where we use the standard convention that all Lie brackets not determined by the relations above are $0$. (The definition actually extends to $\lambda \in \{0,1\}$, but the argument in Lemma \ref{MalcevBasis} does not.)

We recall from \cite[Note, following Thm. 1.1.13]{CorwinGreenleaf} that a basis $(Y_1, \dots, Y_7)$ of $\mathfrak g_\lambda$  is called a \emph{strong Malcev basis} provided that the linear subspaces $\mathfrak h_m := \langle Y_1, \dots, Y_m \rangle \subset \mathfrak g_\lambda$ are ideals for all $m=1, \dots, 7$. 
\begin{lemma}\label{MalcevBasis} Let $Y_1, \dots, Y_7$ be a strong Malcev basis of $\mathfrak g_\lambda$ for some $\lambda \in \R \setminus \{0,1\}$. Then there exist $\mu_1, \dots, \mu_7 \in \R^\times$ such that
\[
Y_1 = \mu_7 X_7, \; Y_2 = \mu_6 X_6, \;\dots,\; Y_6 = \mu_2X_2,\; Y_7 = \mu_1 X_1.
\]
\end{lemma}
\begin{proof} If $Y_1, \dots, Y_7$ is a strong Malcev basis, then $[X_1, Y_1] \in \R \cdot Y_1$. Since
\[
\left[X_1, \sum_{j=1}^7 \mu_j X_j \right] = \sum_{j=1}^6 \mu_j X_{j+1},
\]
this is only possible if $Y_1 = \mu_7X_7$ for some $\mu_7 \in \R^\times$ and thus $\mathfrak h_1 = \R \cdot Y_7$. Similarly the condition $[X_1, \mathfrak h_2/\mathfrak h_1] \subset \mathfrak h_2/\mathfrak h_1$ implies that $Y_2 = \mu_6 X_6$, and recursively we obtain the condition of the lemma.
\end{proof}
Given $\lambda \in \R \setminus\{0,1\}$ let us denote by $G_\lambda$ the unique simply-connected Lie group with Lie algebra $\mathfrak g_\lambda$.
\begin{proposition}\label{NoLattice} The Lie group $G_\lambda$ admits a (uniform) lattice if and only if $\lambda \in \mathbb Q \setminus\{0,1\}$.
\end{proposition}
\begin{proof} By \cite[Thm. 5.1.8]{CorwinGreenleaf} the group $G_\lambda$ admits a (uniform) lattice if it admits a basis with rational structure constants. If $\lambda \in \Q\setminus\{0,1\}$, then $X_1, \dots, X_7$ is such a basis.
For the converse, we observe that it follows from \cite[Thm. 5.1.6 and Thm. 5.1.8]{CorwinGreenleaf} that if $G_\lambda$ contains a (uniform) lattice, then not only does it admit a basis with rational structure constants, but even a strong Malcev basis with rational structure constants. By Lemma \ref{MalcevBasis} this implies that there exist $\mu_1, \dots, \mu_7 \in \R^\times$ such that
\[
[\mu_i X_i, \mu_j X_j] \in \Q \quad (1 \leq i,j \leq 7),
\]
which amounts to
\begin{equation}\label{Rationality}
\left\{\frac{\mu_1\mu_2}{\mu_3}, \frac{\mu_1\mu_3}{\mu_4},  \frac{\mu_1\mu_4}{\mu_5},\frac{\mu_1\mu_5}{\mu_6},  \frac{\mu_1\mu_6}{\mu_7},   \frac{\mu_2\mu_3}{\mu_5},   \frac{\mu_2\mu_4}{\mu_6},   \frac{\mu_2\mu_5}{\lambda \mu_7},   \frac{\mu_3\mu_4}{(1-\lambda)\mu_7}\right\} \subset \Q.
\end{equation}
Thus if we write $a \equiv_\Q b$ to denote that $a,b\in \R^\times$ are rational multiples of each other, then we have
\[
\mu_7  \equiv_\Q \mu_1\mu_6 \equiv_\Q \mu_1\mu_2\mu_4 = \mu_2 \mu_1\mu_4  \equiv_\Q \mu_2\mu_5  \equiv_\Q \lambda \mu_7,
\]
hence $\lambda  \equiv_\Q 1$, i.e.\ $\lambda$ is rational.
\end{proof}
\begin{corollary}\label{Bigger} Let $\lambda \in \overline{\Q} \setminus \Q$. Then $G_\lambda$ admits a uniform model set, but does not admit any lattice. 
\end{corollary}
\begin{proof} That $G_\lambda$ does not admit a lattice was already established in Proposition \ref{NoLattice}. Conversely, let $K := \Q(\lambda)$ and let $n := [K:\Q] < \infty$.
Since the structure constants with respect to the basis $(X_1, \dots, X_7)$ of $\mathfrak g_\lambda$ are contained in $K$, the group $G_\lambda = \mathbb G(\R)$ is given by the real points of an algebraic group $\mathbb G$ defined over $K$. Let $\mathbb H := {\rm Res}_{K/\Q}\mathbb G$ denote the Weil restriction of $\mathbb G$ to $\Q$. Then $\mathbb H$ is defined over $\Q$, and there exists a real Lie group $H$ such that  $\mathbb H(\R) = G_\lambda \times H$. Since $\mathbb H$ is defined over $\Q$ the group $G_\lambda \times H$ contains a uniform lattice $\Gamma$, and one can show that this lattice is irreducible. It thus gives rise to a uniform model set in 
$G_\lambda$ via the construction in Example \ref{Ex2Intro} (by choosing appropriate compact windows in $H$).
\end{proof}
The same argument applies to any simply-connected nilpotent algebraic group, which is defined over a number field, but not over $\Q$. However, since writing out such examples explicitly in higher dimension gets complicated very quickly, we confine ourself to one more example with interesting additional properties.
\begin{example}\label{Elek} A class of $8$-dimensional nilpotent Lie algebras $\mathfrak g_d$, parametrized by $d \in \R$, is given by generators $\{X_1, \dots, X_5, Y_1, Y_2\}$ and bracket relations
\[
[X_1, X_2] = [X_3, X_4] = Y_1,\quad [X_3, X_5] = [X_6, X_4] = Y_2, \quad [X_5, X_6] = dY_1,
\]
where again we use the convention that all Lie brackets not determined by the relations above are $0$. One can again show, that the corresponding simply-connected Lie group $G_d$ admits a (uniform) lattice if and only if $d \in \Q$, though the proof is more involved (see \cite{Scheunemann}). On the other hand, one shows just as in Corollary \ref{Bigger} that if $d \in \overline{\Q}$, then $G_d$ admits a uniform model set, and hence a uniform approximate lattice. This example is remarkable for the following reason: It was established in \cite[Prop. 2]{ElekTardos} that for $d \not \in \Q$ the group $G_d$ is not quasi-isometric to any finitely-generated group (see Subsection \ref{SecQIclass} below for a discussion of the canonical QI class of a compactly-generated group), in fact not even to any vertex-transitive graph. On the other hand, we will see in Theorem \ref{MilnorSchwarz} below that for $d \in \overline{\Q}\setminus \Q$ it is quasi-isometric to a finitely-generated \emph{approximate} group, namely any of its uniform approximate lattices.
\end{example}

\subsection{Finite generation}
In analogy with the group case we say that an approximate group $(\Lambda, \Lambda^\infty)$ is \emph{finitely generated} if $\Lambda^\infty$ is finitely generated as an abstract group. 
Recall that a uniform lattice in a lcsc group $G$ is finitely generated if and only if $G$ is compactly generated \cite[Prop. 5.C.3]{CdlH}. This statements generalizes to uniform approximate lattices. 
\begin{theorem}\label{ThmFiniteType}\label{CompatibleGeneratingSets} Assume $\Lambda \subset G$ is a uniform approximate lattice. Then the following are equivalent.
\begin{enumerate}[(i)]
\item $G$ is compactly generated.
\item $\Lambda$ is finitely generated.
\end{enumerate}
In this case, there exists a compact subset $K \subset G$ such that
\[
G=\langle K \rangle \quad \text{and} \quad \langle \Lambda \rangle = \langle \Lambda^2 \cap K \rangle.
\]
\end{theorem}
\begin{proof} The implication (ii) $\Rightarrow$ (i) of the proposition is obvious:  If $F$ is a finite generating set for $\Lambda^\infty$ and $K \subset G$ is compact with $G = \Lambda K$, then $F \cup K$ is a compact generating set for $G$. Conversely assume that (i) holds. We fix a proper left-invariant metric $d$ inducing the given topology on $G$. By \cite[Prop. 4.B.8]{CdlH} the space $(G,d)$ is coarsely connected, i.e.\ there exists $C>0$ such that for every $x,y$ in $G$ there exist points $x = x_0, x_1, \dots, x_n = y$ in $G$ with $d(x_i, x_{i+1}) \leq C$ for each $i\in \{0, \dots, n-1\}$. Also, by assumption $\Lambda$ is $R$-relatively dense for some $R > 0$. Let now $K$ be any compact generating set of $G$ containing the ball of radius $C+2R$ around $e$. We claim that 
every element $g \in \Lambda$ is a finite product of elements in $F := \Lambda^2 \cap K$. This will imply that $F$ generates $\Lambda^\infty$ and since $F$ is finite by Proposition \ref{PropDiscreteEquiv} this will finish the proof.

Given $g \in \Lambda$ we choose elements $e = x_0, x_1, \dots, x_n = g$ in $G$ such that $d(x_i, x_{i+1}) \leq C$. Using relative density of $\Lambda$  we choose  $y_i \in B_R(x_i)$ such that $y_i \in \Lambda$. We may assume $y_0 = x_0 = e$ and $y_n = x_n = g$. For every $i\in \{0, \dots, n-1\}$ we then have
\[
d(y_i^{-1}y_{i+1}, e) = d(y_i, y_{i+1}) < C+2R.
\]
On the other hand, $y_i^{-1}y_{i+1} \in \Lambda^2$, hence $y_i^{-1}y_{i+1} \in \Lambda^2 \cap B_{C+2R}(e) \subset F$. Now, using that $y_0 = e$ and that $y_i^{-1}y_{i+1} \in F$ we get
\[
g = y_n = y_{n-1}(y_{n-1}^{-1}y_n) = \dots = (y_0^{-1}y_1)(y_1^{-1}y_2) \cdots (y_{n-1}^{-1}y_n) \in F^n.
\]
This finishes the proof.
\end{proof}

\section{Some geometric approximate group theory}\label{SecAGGT}

\subsection{The canonical QI class of a finitely-generated approximate group} \label{SecQIclass}

Finitely generated groups carry a distinguished quasi-isometry class of metrics, whose study is the subject of geometric group theory. Here we propose a generalization of geometric group theory to finitely generated approximate groups. To this end, let us fix our notation concerning quasi-isometries.


Given metric spaces $(X, d_X), (Y, d_Y)$ a map $f: X \to Y$ is called a \emph{$(K,C)$-quasi-isometric embedding} if for all $x_1, x_2 \in X$,
\[
\frac{1}{K} \cdot (d_X(x_1,x_2) - C) \leq d_Y(f(x_1), f(x_2)) \leq K \cdot d_X(x_1,x_2)+C.
\]
It is called a \emph{$(K,C)$-quasi-isometry} if moreover $N_R(f(X)) = Y$ for some $R>0$. In this case there exists a \emph{quasi-inverse} $\overline{f}: Y \to X$, i.e.\ a quasi-isometry such that 
\[
\sup_{x \in X} d_X(x, \overline{f}(f(x))) < \infty \quad \text{and} \quad \sup_{y\in Y} d_Y(y, f(\overline{f}(y))) < \infty.
\]
\begin{example}
The inclusion of a relatively dense subset into a metric space is a quasi-isometry. 
In particular, every Delone set is quasi-isometric to its ambient space.
\end{example}
If $\Gamma$ is a finitely generated group, then any two word metrics with respect to finite generating sets of $S$ are quasi-isometric. We refer to the common quasi-isometry (QI) class of these metrics as the \emph{canonical QI class} of $\Gamma$. We need two generalizations of this concept.

Firstly, let $G$ be a compactly generated lcsc group. By \cite[Prop. 4.B.4]{CdlH} any two word metrics with respect to compact generating sets on $G$ are quasi-isometric, and we refer to their common QI class as the \emph{canonical QI class} of $G$. 

Secondly, note that if a set $X$ is equipped with a QI class $[d]$ of metrics and $Y\subset X$ is an arbitrary non-empty subset, then any two metrics in $[d]$ restrict to quasi-isometric metrics on $Y$, hence the restriction $[d]|_Y := [d|_Y]$ is a well-defined QI class on $Y$. In particular, if $(\Lambda, \Lambda^\infty)$ is a finitely generated approximate group, then the canonical QI class of $\Lambda^\infty$ restricts to a QI class of metrics on $\Lambda$ which we refer to as the \emph{canonical QI class} of $\Lambda$.

\subsection{A Milnor-Schwarz lemma for uniform approximate lattices}\label{SecMS} A fundamental theorem of geometric group theory is the Milnor-Schwarz Lemma. A classical version of this theorem can be stated as follows \cite[Prop. I.8.19]{BridsonHaefliger}:
\begin{theorem}[Milnor--Schwarz Lemma, geometric version]\label{MSClassic} Let $X$ be a proper geodesic metric space and let $\Gamma$ be a group. If $\Gamma$ acts properly and cocompactly on $X$, then $\Gamma$ is finitely-generated and for every $x \in X$ the orbit map $\Gamma \to X$, $\gamma\mapsto x$ is a quasi-isometry with respect to the canonical QI class of $\Gamma$.
\end{theorem}
It is possible to reformulate this theorem in terms of lcsc groups. The key point here is that the isometry group ${\rm Is}(X)$ of a proper metric space carries a natural lcsc group topology. Namely, the compact-open topology, the topology of pointwise convergence or uniform convergence on compact sets all coincide on ${\rm Is}(X)$ and turn the latter into a lcsc group \cite[Prop. 5.B.5]{CdlH}. In fact, by a theorem of Malicki and Solicki  \cite[Thm. 5.B.14]{CdlH}, every lcsc group is isomorphic to  ${\rm Is}(X)$ for a suitable proper metric space $X$. Moreover, if $X$ is geodesic then ${\rm Is}(X)$ is compactly generated \cite[Thm. 4.C.5]{CdlH}. A proper cocompact action of $\Gamma$ on $X$ corresponds to a homomorphism with finite kernel into ${\rm Is}(X)$ whose image is a uniform lattice. Thus Theorem \ref{MSClassic} admits the following reformulation:
\begin{theorem}[Milnor--Schwarz Lemma, group theoretical version] Let $G$ be a compactly generated lcsc group and let $\Gamma < G$ be a uniform lattice. Then $\Gamma$ is finitely generated and the canonical QI class of $G$ restricts to the canonical QI class of $\Gamma$. 
\end{theorem}
We generalize this formulation to uniform approximate lattices.
\begin{theorem}[Milnor--Schwarz Lemma for approximate groups]\label{MilnorSchwarz} Let $G$ be a compactly generated lcsc group and let $\Lambda \subset G$ be a uniform approximate lattice. Then $\Lambda^\infty$ is finitely generated and the canonical QI class of $G$ restricts to the canonical QI class of $\Lambda$. 
\end{theorem}
The following example illustrates the theorem and shows that the inclusion $\Lambda^\infty \to G$ is in general not a quasi-isometry. 
\begin{example} Let $(-)^*: \Z[\sqrt 2] \to \Z[\sqrt 2]$ denote the Galois conjugation $(a+b\sqrt 2)^* := a-b\sqrt 2$ and let $\Lambda \subset \Lambda^\infty \subset \R^2$ be given by
\[
 \Lambda^\infty := \{(a, a^*) \in \R^2\mid a \in \Z[\sqrt 2]\} \supset \Lambda := \{(a, a^*) \in \R^2\mid a \in \Z[\sqrt 2], |a^*|\leq5\}.
\]
Then $(\Lambda, \Lambda^\infty)$ is an approximate group. Since $\Lambda^\infty$ is a lattice in $\R^2$, its canonical quasi-isometry class is represented by the restriction of the Euclidean metric from $\R^2$, an its subset $\Lambda$ is easily seen to be quasi-isometric to $\R$. Projection onto the first coordinate defines an embedding $\pi_1: \Lambda^\infty \hookrightarrow \R$, and $\pi_1(\Lambda)$ is a uniform approximate lattice in $\R$ by Proposition \ref{ModelSet}. In particular, the restriction $\pi_1|_\Lambda: \Lambda \to \R$ is a quasi-isometry by Theorem \ref{MilnorSchwarz}, as can also be checked easily directly in this case. Note that the map $\pi_1: \Lambda^\infty \to \R$ is not a quasi-isometry, since points with  a small first and large second component are large in $ \Lambda^\infty$ but have small image in $\R$. 
\end{example}
The proof of Theorem \ref{MilnorSchwarz} uses a version of Gromov's ``Trivial Lemma'' \cite[0.2.D]{Gromov} which we state in the following convenient form. (For a proof see e.g. \cite[Prop. 3.B.9]{CdlH}.) Here a space is called \emph{large-scale geodesic} if it is quasi-isometric to a geodesic metric space. 
\begin{lemma}[Gromov]\label{GromovLemma}
Let $X$ be a large-scale geodesic metric space, $Y$ an arbitrary metric space and $f: X \to Y$ a map. Assume that there exists a non-decreasing function $\rho: [0, \infty) \to [0, \infty)$ such that for all $x_1, x_2 \in X$.
\[
d_Y(f(x_1), f(x_2))  \leq \rho(d_X(x_1,x_2)).
\] 
Then there exist constants $C>1$, $D>0$ such that for all $x_1, x_2 \in X$ we have
\begin{equation*}
d_Y(f(x_1), f(x_2))  \leq C \cdot d_X(x_1,x_2) + D. \qed
\end{equation*}
\end{lemma}
\begin{proof}[Proof of Theorem \ref{MilnorSchwarz}] That $\Lambda^\infty$ is finitely generated was already established in Theorem \ref{ThmFiniteType}. Moreover, according to this theorem we can find a compact generating set $K$ of $G$ with the property that $F :=K \cap \Lambda^2$ generates $\Lambda^\infty$. We may assume that $K$, and hence $F$, contains the identity so that $\Lambda^\infty$ is the ascending union of the sets $F^k$.

Let us denote the word lengths with respect to $K$ and $F$ by $\|-\|_K$ and $\|-\|_F$ respectively, and by $d_K$ and $d_F$ the corresponding left-invariant word metrics. Since $\Lambda$ is a Delone set in $G$, the inclusion $(\Lambda, d_K|_\Lambda) \hookrightarrow (G, d_K)$ is a quasi-isometry. We thus need to show only that the identity map
\[\iota: (\Lambda, d_K|_\Lambda) \to (\Lambda, d_F|_\Lambda)\]
is a quasi-isometry. Since $F \subset K$ we have $\|g\|_F \geq \|g\|_K$ for all $g \in \Lambda^\infty$ and thus $d_K(x_1,x_2) \leq d_F(x_1,x_2) = d_F(\iota(x_1), \iota(x_2)) $ for all $x_1,x_2 \in \Lambda$. It thus remains to show that there exist $C>1$, $D>0$  such that for all $x_1, x_2 \in \Lambda$,
\[
d_F(\iota(x_1), \iota(x_2))  \leq C \cdot d_K(x_1,x_2)) + D
\]
By \cite[Prop. 4.B.4]{CdlH} the space $(G, d_K)$ is large-scale geodesic. Since this property is a quasi-isometry invariant, we deduce that also $(\Lambda, d_K|_\Lambda)$ is large-scale geodesic. Consequently, Lemma \ref{GromovLemma} applies and we are reduced to showing that there exists a non-decreasing function $\rho: [0, \infty) \to [0, \infty)$  such that for all $x_1, x_2 \in \Lambda$,
\begin{equation}\label{MSToShow}
d_F(\iota(x_1), \iota(x_2)) = d_F(x_1,x_2) \leq \rho(d_K(x_1,x_2)).
\end{equation}
Now for every $n \in \mathbb N$ the set $K^n$ is compact, and hence $K^n \cap \Lambda$ is finite. Since $\Lambda$ is contained in the ascending union $\bigcup F^k$, every finite subset of $\Lambda$ is contained in one of these set. We thus find a non-descreasing function $\rho: \mathbb N \to \mathbb N$ such that
\[
K^n \cap \Lambda \subset F^{\rho(n)}.
\]
Thus if $g \in \Lambda$ and $n:= \|g\|_K $, then $g \in K^n \cap \Lambda \subset F^{\rho(n)}$ and thus $\|g\|_F \leq \rho(n) = \rho(\|g\|_K)$. Since $g$ was arbitrary we deduce that for all $g \in G$ we have $\|g\|_F \leq \rho(\|g\|_K)$ which establishes \eqref{MSToShow} and finishes the proof.
\end{proof}

\subsection{The left-regular quasi-action}

Let $(X,d)$ be a metric space. Since the composition of a  $(K_1, C_1)$-quasi-isometry and a $(K_2, C_2)$-quasi-isometry from $X$ to $X$ is a $(K_1K_2, K_1C_2+K_2C_1)$-quasi-isometry, the set $\widetilde{\rm QI}(X)$ of all self-quasi-isometries of $X$ is closed under composition. We call $f,g \in \widetilde{\rm QI}(X)$ equivalent, denoted $f\sim g$, provided
\[
\sup_{x \in X} \; d(f(x), g(x)) < \infty.
\]
Then ${\rm QI}(X) := \widetilde{\rm QI}(X)/\sim$ is a group under composition of representatives, called the \emph{quasi-isometry group} of $X$. Implicitly the group ${\rm QI}(X)$ appears in many applications in geometric group theory. Explicitly it appears e.g. in \cite{Furman, Dymarz}. Given $f \in \widetilde{\rm QI}(X)$ we denote by $[f]$ the equivalence class of $f$ in ${\rm QI}(X)$. We say that a class in ${\rm QI}(X)$ is a $(K,C)$-quasi-isometry class if it can be represented by a $(K,C)$-quasi-isometry. A subset $A \subset {\rm QI}(X)$ is called \emph{uniform} if there exist $K, C$ such that every class in $A$ is a $(K,C)$-quasi-isometry class.
\begin{definition} Let $\Gamma$ be a group and $(\Lambda, \Lambda^\infty)$ be an approximate group. A homomorphism $\rho: \Gamma \to {\rm QI}(X)$ is called a  \emph{quasi-action} of $\Gamma$ on $X$ if $\rho(\Gamma)$ is a uniform subset of ${\rm QI}(X)$. More generally, a homomorphism $\rho: \Lambda^\infty \to {\rm QI}(X)$ is called a \emph{quasi-action} of $(\Lambda, \Lambda^\infty)$ on $X$ if $\rho(\Lambda)$ is a uniform subset of ${\rm QI}(X)$.
\end{definition}
For groups, this is the standard definition of a quasi-action, see e.g. \cite{Quasiaction}.
\begin{example}
If $G$ is a compactly-generated lcsc group, then for every $g \in G$ the left-multiplication $\lambda_g: h \mapsto gh$ define a quasi-isometry of $G$ with respect to the canonical QI class on $G$, and the homomorphism $\lambda_G: G \to {\rm QI}(G)$, $g \mapsto [\lambda_g]$ has uniform image, hence defines a quasi-action.
\end{example}
\begin{definition}\label{DefLRQAGroup} Given a compactly-generated lcsc group $G$, the quasi-action $\lambda_G: G \to {\rm QI}(G)$, $g \mapsto [\lambda_g]$ 
is called the \emph{left-regular quasi-action} of $G$.
\end{definition}
\begin{remark}\label{QAnotfaithful} 
Given a metric space $X$, the canonical group homomorphism ${\rm Is}(X) \to {\rm QI}(X)$ need not be injective. For example, the the image of ${\rm Is}(\R) \cong \R \rtimes \Z/2\Z$ in ${\rm QI}(\R)$ collapes to $\Z/2\Z$. For the same reason, the left-regular quasi-action of a compactly-generated lcsc group need not be faithful. For example, the left-regular quasi-action of an abelian group is always trivial. One can show that in general the kernel of the left-regular quasi-action is given by the elements whose conjugacy class is bounded. 
\end{remark}
We now extend the definition of the left-regular quasi-action to finitely-generated approximate groups.
\begin{proposition}\label{LRQA}  Let $(\Lambda, \Lambda^\infty)$ be a finitely-generated approximate group and $S \subset \Lambda^\infty$ a finite generating set. There exists a unique quasi-action $\lambda: \Lambda^\infty \to {\rm QI}(\Lambda)$ and
(non-unique) representatives $\lambda_g \in \lambda(g)$ with the following properties.
\begin{enumerate}[(i)]
\item For every $g \in \Lambda^\infty$,
\[
D(g) := \sup_{h \in \Lambda} d_S(\lambda_g(h), gh) < \infty,
\]
and for every $k \geq 1$ the constant $D(g)$ is bounded uniformly over $\Lambda^k$.
\item For every $k \geq 1$ the subset $\lambda(\Lambda^k) \subset {\rm QI}(\Lambda)$ is uniform.
\item If $k,l \geq 1$ then $\lambda_g \lambda_h$ is at uniformly bounded distance from $\lambda_{gh}$ as $g$ ranges over $\Lambda^k$ and $h$ ranges over $\Lambda^l$.
\end{enumerate}
\end{proposition}
Note that in the group case $(\Lambda, \Lambda^\infty) = (\Gamma, \Gamma)$ the quasi-acion $\lambda$ is just the left-regular quasi-action. We thus define:
\begin{definition}\label{DefLRQA} The quasi-action $\lambda: \Lambda^\infty \to {\rm QI}(\Lambda)$, $g \mapsto [\lambda_g]$ is called the \emph{left-regular quasi-action} of the finitely-generated approximate group $(\Lambda, \Lambda^\infty)$. 
\end{definition}

Concerning the proof of Proposition \ref{LRQA}, it is clear that $\lambda$ is uniquely determined by (i). To show existence of $\lambda$, we choose a  finite set $F \subset \Lambda^\infty$ such that $\Lambda^2 \subset \Lambda F$ and define 
\[
\delta := \max_{f \in F} d_S(e, f).
\]
We consider each of the sets $\Lambda^n \subset \Lambda^\infty$ as a metric space with respect to the restriction of $d_S$ so that the inclusions $\Lambda \subset \Lambda^2 \subset \Lambda^3 \subset \dots$ are isometric. Note that since $\Lambda^{n+1} \subset \Lambda^nF$, every $x \in \Lambda^{n+1}$ can be written as $x = x'f$ with $f \in F$ and $x' \in \Lambda^n$. Thus
\[
d(x,x') = d(x'f, x')  = d(f,e) \leq \delta,
\]
i.e.\ $\Lambda^{n}$ is $\delta$-relatively dense in $\Lambda^{n+1}$ and the inclusion $\iota_n^{n+1}: \Lambda^{n} \to \Lambda^{n+1}$ is a quasi-isometry. We can thus find a map $p^{n+1}_n: \Lambda^{n+1} \to \Lambda^n$ such that \[
p_n^{n+1} \iota^{n+1}_n(x) = x \text{ for all } x \in \Lambda^n  \quad \text{ and } \quad d(\iota_n^{n+1}(p^{n+1}_n(x)), x) \leq \delta \text{ for all }x \in \Lambda^{n+1}.\]
We thus have a tower of $(1, 2\delta)$-quasi-isometries
\[\begin{xy}\xymatrix{
\Lambda \ar@/^/@/^/[r]^{\iota_1^2}& \ar@/^/[l]^{p_1^2}\Lambda^2 \ar@/^/[r]^{\iota_2^3}& \ar@/^/[l]^{p_2^3}\Lambda^3 \ar@/^/[r]^{\iota_3^4}&  \ar@/^/[l]^{p_3^4}\Lambda^4 \ar@/^/[r]^{\iota_4^5}& \ar@/^/[l]^{p_4^5}\Lambda^5 \ar@/^/[r]^{\iota_5^6}& \ar@/^/[l]^{p_5^6} \Lambda^6 \ar@/^/[r]^{\iota_6^7}&  \ar@/^/[l]^{p_6^7} \dots
}\end{xy}
\]
In the sequel we denote for all $k<n$ by $\iota_k^n: \Lambda^k \to \Lambda^n$ the isometric embedding and define
\[
 p_{k}^n: \Lambda^n \to \Lambda^k, \quad x \mapsto p^n_{n-1} \circ \dots \circ p^{k+1}_k(x).
\]
Then for all $k<n$ the maps $\iota_k^n$ and $p_k^n$ are $(1, 2(n-k)\delta)$-quasi-isometries and satisfy \[p^n_k(\iota_k^n(x)) = x.\] Now let $g \in \Lambda^k$ and $x \in \Lambda$. Then $gx \in \Lambda^{k+1}$, and for every $n \geq k+1$ we have
\[
p^{n}_1(gx) = p^n_1 \iota_{k+1}^n(gx) =  p_1^{k+1}p^{n}_{k+1}\iota_{k+1}^n(gx)=p^{k+1}(gx).
\]
Thus $\lambda_g(x) := p^{n}(gx)$ is well-defined independent of $n$ as long as $n$ is sufficiently large. Moreover, if $g \in \Lambda^k$ then $\lambda_g: \Lambda \to \Lambda$ is a $(1, 2k\delta)$-quasi-isometry, which we can visualize as
\[\begin{xy}\xymatrix{
\lambda_g: &\Lambda \ar@/^1pc/[rrrr]^{g\cdot}& \ar@/^/[l]^{p_1^2}\Lambda^2& \ar@/^/[l]^{p_2^3}\Lambda^3&  \ar@/^/[l]^{p_3^4}
 \dots & \ar@/^/[l]^{p_{k-1}^k} \Lambda^{k+1}
}\end{xy}.
\]
\begin{proof}[Proof of Proposition \ref{LRQA}] Set $\lambda(g) := [\lambda_g]$. Then (i) and (ii) hold by our previous discussion. As for (iii), given $g \in \Lambda^k$, $h \in \Lambda^l$ and $x \in \Lambda$ we have
\begin{eqnarray*}
d(\lambda_{gh}(x), \lambda_g\lambda_h(x)) &=& d(p_1^{k+l+1}(ghx), p_1^{k+1}(gp_1^{l+1}(hx)))\\
&\leq& d(p_1^{k+l+1}(ghx), ghx) + d(ghx, gp_1^{l+1}(hx)) \\&& + d(gp_1^{l+1}(hx), p_1^{k+1}(gp_1^{l+1}(hx)))\\
&\leq& 2(k+l)\delta +  2l\delta + 2k\delta\\
&\leq& 4(k+l)\delta,
\end{eqnarray*}
This implies in particular that $[\lambda_{gh}] = [\lambda_g][\lambda_h]$, which shows that $\lambda$ is a homomorphism and finishes the proof.
\end{proof}


\subsection{The QI-rigidity problem} \label{SecQI} 

Let $G$ be a compactly generated lcsc group. A homomorphism from a finitely generated group $\Gamma$ to $G$ is called \emph{geometric} if it has finite kernel and its image is a uniform lattice in $G$. By a variation of the Milnor-Schwarz lemma, if $\rho: \Gamma \to G$ is geometric, then $\Gamma$ is quasi-isometric to $G$. The QI-rigidity problem asks for a converse to this statement. There are various different inequivalent versions in the literature. We will generalize the following weak notion.
\begin{definition} A lcsc group $G$ is called \emph{QI-rigid} (with respect to groups) if every finitely generated group $\Gamma$ quasi-isometric to $G$ admits a geometric homomorphism into $G$. 
\end{definition}
We would like to formulate an analogous notion of QI-rigidity with respect to approximate groups. A subtle point concerns the definition of geometric homomorphism. It turns out that allowing for finite kernels is not good enough:
\begin{example} Let $\Gamma_1$ be a finitely-generated uniform lattice in a lcsc group $G$, $\Gamma_2$ a finitely generated group  which does not embed in $G$ and $S \subset \Gamma_2$ a finite generating set. Then $(\Lambda, \Lambda^\infty) := (\Gamma_1 \times S, \Gamma_1 \times \Gamma_2)$ is a finitely-generated approximate group quasi-isometric to $G$, but there is no embedding of $\Lambda^\infty$ into $G$. However, projection to the first factor provides a map $\pi: \Lambda^\infty \to G$ with infinite kernel which restricts to a quasi-isometry $\Lambda \to G$. Note that the induced map $\Lambda \to \pi(\Lambda)$ is a quasi-isometry with respect to the respective canonical QI classes.
\end{example}
In view of this example we define:
\begin{definition} Let $ (\Lambda, \Lambda^\infty)$ be a finitely generated approximate group, $ (\Xi, \Xi^\infty)$ an arbitrary approximate group and $\rho: (\Lambda, \Lambda^\infty) \to (\Xi, \Xi^\infty)$ a homomorphism. We say that $\rho$ has \emph{small kernel} if the induced surjection $\Lambda \to \rho(\Lambda)$ is a quasi-isometry with respect to the canonical QI classes of $(\Lambda, \Lambda^\infty)$ and $(\rho(\Lambda), \rho(\Lambda)^\infty)$.
\end{definition}
\begin{definition} Let $G$ be a lscs group and  $(\Lambda, \Lambda^\infty)$ a finitely generated approximate group.
A \emph{geometric representation} of a finitely-generated approximate group $(\Lambda, \Lambda^\infty)$ is a representation $\rho: \Lambda^\infty \to G$ with small kernel such that $\varphi(\Lambda)$ is a uniform approximate lattice in $G$.
\end{definition}
It is immediate from Theorem \ref{MilnorSchwarz} that every geometric representation $\varphi: \Lambda^\infty \to G$ restricts to a quasi-isometry $\Lambda \to G$.
\begin{definition} A lcsc group $G$ is \emph{QI-rigid with respect to approximate groups} if every finitely-generated approximate group quasi-isometric to $G$ admits a geometric representation into $G$.
\end{definition}
\begin{problem} Find examples of lcsc groups $G$ which are QI-rigid with respect to approximate groups.
\end{problem}
To show that our definition of QI-rigidity is meaningful we extend the celebrated QI-rigidity results of Kleiner and Leeb \cite{KleinerLeeb} concerning higher rank symmetric spaces of non-compact type to the setting of approximate groups. We expect that other classical QI-rigidity results can be extended to the approximate setting along similar lines.
\begin{definition} Let $(X, d)$ be a proper locally compact geodesic metric space and assume that ${\rm Is}(X,d)$ acts cocompactly on $X$. Then $X$ is called a \emph{Kleiner--Leeb (KL) space} if for every $K \geq 1$, $C>0$ there exists $D>0$ such that every $(K,C)$-quasi-isometry $f$ of $X$ there exists a unique isometry $\widehat{f}$ of $(X,d)$ such that $d(f(g), \widehat{f}(g)) < D$ for all $g \in G$.
\end{definition}
Note that the KL condition implies that the inclusion ${\rm Is}(X, d) \to {\rm QI}(X)$ is an isomorphism; however, it is slightly stronger in that it demands that the constant $D$ in the definition is uniform in the quasi-isometry constants $K$ and $C$. By the theorem of Kleiner and Leeb alluded to earlier, every higher rank symmetric space of non-compact type is a KL-space \cite{KleinerLeeb}, hence the name. Other examples are given by quaternion hyperbolic spaces \cite{Pansu} and certain higher rank Euclidean Tits buildings \cite{KleinerLeeb}. We are going to prove: 
\begin{theorem}\label{QIRigidity} Let $X$ be a Kleiner--Leeb space and $G := {\rm Is}(X)$. Then $G$ is QI-rigid with respect to approximate groups.
\end{theorem}
The first step in the proof of Theorem \ref{QIRigidity} is the following observation. Recall from Definition \ref{DefLRQAGroup} the definition of the left-regular quasi-action of a compactly-generated lcsc group.
\begin{proposition}\label{RLQAKL} If $G$ is the isometry group of a Kleiner--Leeb space, then $G$ is compactly-generated and the left-regular quasi-action $\lambda_G: G \to {\rm QI}(G)$ is an isomorphism.
\end{proposition}
\begin{proof} It follows from \cite[Thm 4.C.5 and Prop. 5.B.10]{CdlH} that $G$ is compactly generated and that the orbit maps $G \to X$ are quasi-isometries. In particular, ${\rm QI}(X) \cong {\rm QI}(G)$.  Moreover, the KL assumption implies that the canonical map ${\rm Is}(X, d) \to {\rm QI}(X)$ is an isomorphism, and the composition of these two isomorphisms coincides with the left-regular quasi-action.
\end{proof}
Now let $G$ be the isometry group of a KL space and let  $(\Lambda, \Lambda^\infty)$ be a finitely-generated approimate group quasi-isometric to $G$. We fix a quasi-isometry $\varphi: \Lambda \to G$ and a quasi-inverse $\overline{\varphi}: G \to \Lambda$ with $\varphi(e) = e$ and $\overline{\varphi}(e) = e$. Denote by $\psi:  {\rm QI}(\Lambda) \to {\rm QI}(G)$ the isomorphism $f \mapsto\varphi \circ f \circ  \overline{\varphi}$. Denote by $\lambda: \Lambda^\infty \to {\rm QI}(\Lambda)$ the left-regular quasi-action of $(\Lambda, \Lambda^\infty)$ (cf. Definition \ref{DefLRQA}). By Proposition \ref{RLQAKL} we then obtain a homomorphism
\[
\rho: \Lambda^\infty \xrightarrow{\lambda} {\rm QI}(\Lambda) \xrightarrow{\psi} {\rm QI}(G) \xrightarrow{\lambda_G^{-1}} G. 
\]
We are going to show that $\rho$ is a geometric representation. For this we have to establish the following three items.
\begin{enumerate}[(GR1)]
\item For every $\gamma \in \Lambda$ the elements $\rho(\gamma)$ and $\varphi(\gamma)$ are at uniformly bounded distance. Since $\varphi(\Lambda)$ is relatively dense in $G$, this implies that $\rho(\Lambda)$ is relatively dense in $G$.
\item $\rho(\Lambda)^3 \subset G$ is uniformly finite, and hence $\rho(\Lambda)$ is an approximate lattice in $G$ by (GR1) and Proposition \ref{PropDiscreteEquiv}.
\item $\rho$ has small kernel.
\end{enumerate}
For the proof of (GR1) we need:
\begin{lemma}\label{QIRigidityMain} For every $k \geq 1$ there exists $D_k > 0$ such that for all $\gamma \in \Lambda^k$,
\[
 \sup_{h \in G} d(\varphi(\lambda_\gamma(\overline{\varphi}(h))), \rho(\gamma)h) < D_k.
\]
\end{lemma}
\begin{proof} By definition $\rho(\gamma)$ is the unique element in $G$ such that 
\[D(\gamma) := \sup_{h \in G} d(\varphi(\lambda_\gamma(\overline{\varphi}(h))), \rho(\gamma)h)< \infty.\]
Moreover, by Proposition \ref{LRQA} the set $\lambda(\Lambda^k) \subset {\rm QI}(\Lambda^k)$ is uniform, and thus the set $\varphi(\lambda(\Lambda^k))$ is uniform. It thus follows from the KL property that $D(\gamma)$ is uniformly bounded as $\gamma$ ranges over $\Lambda^k$.
\end{proof}
\begin{proof}[Proof of (GR1)] Apply Lemma \ref{QIRigidityMain} with $k =1$ to obtain for all $\gamma \in \Lambda$ the inequality
\[
d(\varphi(\gamma), \rho(\gamma)) = d(\varphi(\lambda_\gamma(\overline{\varphi}(e))), \rho(\gamma)e) \leq \sup_{h \in G} d(\varphi(\lambda_\gamma(\overline{\varphi}(h))), \rho(\gamma)h) < D_1.\qedhere
\]
\end{proof}
Continuing towards (GR2) we observe that the map 
\[
p: \Lambda^3 \times \Lambda \to \Lambda \times \Lambda, \quad (\gamma, x) \mapsto (x, \lambda_\gamma(x))
\]
is proper, i.e.\ pre-images of balls are finite. For the proof of this observation we fix a word metric $d$ on $\Lambda^\infty$ and observe that by Proposition \ref{LRQA} there exists $C>0$ such that
 \begin{equation}\label{ConstantC}
 d(\lambda_\gamma(x), \gamma x) < C
 \end{equation} 
 for all $\gamma \in \Lambda^3$ and $x \in \Lambda$. Now let $R>0$ and assume that $(\gamma, x) \in \Lambda^3 \times \Lambda$ satisfies
 \[
 p(\gamma, x) = (x, \lambda_\gamma(x)) \in (B_R(e) \cap \Lambda) \times (B_R(e) \cap \Lambda),
 \]
i.e.\ $d(e,x) < R$ and $ d(e, \lambda_\gamma(x)) <R$. From the latter inequality and \eqref{ConstantC} we deduce that
\[
d(x, \gamma^{-1}) =  d(e, \gamma x) \leq d(e, \lambda_\gamma(x))  + d(\lambda_\gamma(x), \gamma x)< R + C,
\]
and hence $\gamma \in B_{R+C}(x)^{-1}$. We deduce that for every $R>0$ the set
\[
p^{-1}\left((B_R(e) \cap \Lambda) \times (B_R(e) \cap \Lambda)\right) \subset B_R(e) \times \left(\bigcup_{x \in B_R(e)} B_{R+C}(x)^{-1}\right)
\]
is finite, and hence $p$ is proper as claimed.
\begin{lemma}\label{QIRLemma} Let $\Xi := \rho(\Gamma)^3 = \rho(\Lambda^3)$ and equip $\Xi$ with the discrete topology. Then the map $\Xi \times G \to G \times G$ given by $(\xi, g) \mapsto (g, \xi.g)$ is proper. 
\end{lemma}
\begin{proof} By properness of $p$ we know that for every $R>0$ the set
\[
F_R(\Lambda^3) := \{\gamma \in \Lambda^k \mid B_R(e) \cap \lambda_\gamma(B_R(e)) \neq \emptyset\}
\]
is finite, and we need to show that for every $N \in \mathbb N$ the set
\[
F_N(\Xi) := \{\xi \in \Xi \mid B_N(e) \cap \xi.B_N(e) \neq \emptyset\}
\]
is finite. 
Choose $(R, C)$ such that $\varphi, \overline{\varphi}$ are $(R, C)$-quasi-isometries and their compositions are bounded by $C$, and let $D_k$ as in Lemma \ref{QIRigidityMain}. Finally, let $N \in \mathbb N$ and $\xi \in F_N(\Xi)$ and define
\[
C' := \max\{RN+C, R(N+D_3)+2C\}.
\]

Since $\rho: \Lambda^3 \to \Xi$ is surjective we then find $\gamma \in \Lambda^3$ such that $\xi = \rho(\gamma)$, and since $\xi \in F_N(\Xi)$ we have $B_N(e) \cap \rho(\gamma)B_N(e) \neq \emptyset$. We thus find $g \in G$ such that $d(g,e) < N$ and $d(\rho(\gamma^{-1})g, e) < N$. The former inequality implies that $d(\overline{\varphi}(g),e) < RN+C \leq C'$, and we deduce from the latter inequality and Lemma \ref{QIRigidityMain} that
\[
d(\varphi(\lambda_{\gamma^{-1}}(e)), e) \leq d(\varphi(\lambda_{\gamma^{-1}}(\overline{\varphi}(e))), \rho(\gamma^{-1})g) + d(\rho(\gamma^{-1})g, e) < N+D_3,
\]
and hence
\[
d(\lambda_{\gamma^{-1}}(e), e) \leq d(\overline{\varphi}(\varphi(\lambda_{\gamma^{-1}}(e))), \overline{\varphi}(e))+C < R(N+D_3)+2C \leq C'.
\]
Combining these two observatione we conclude that
\[
\max\{d(\overline{\varphi}(g),e), d(\lambda_{\gamma^{-1}}(\overline{\varphi}(g))), e)\} \leq C',
\] 
and hence $\overline{\varphi}(g)  \in B_{C'}(e) \cap \lambda_\gamma(B_{C'}(e))$ , which implies $\gamma \in F_{C'}(\Lambda^3)$. This shows that $F_N(\Xi) \subset \rho(F_{C'}(\Lambda^3))$ is finite and finishes the proof.
\end{proof}
\begin{proof}[Proof of (GR2)] By Lemma \ref{QIRLemma} the map $\rho(\Lambda)^3 \times G \to G \times G$, $(\gamma, g) \mapsto (g, \gamma g)$ is proper. It follows that for every $R>0$ the set
\[
\rho(\Lambda)^3 \cap B_R(e) = \{\xi \in  \rho(\Lambda)^3\mid \xi \in B_R(e)\} \subset \{\xi \in \rho(\Lambda^3) \mid B_{2R}(e) \cap \xi B_{2R}(e) \neq \emptyset\}
\]
is finite, i.e.\ the set $\rho(\Lambda)^3$ is locally finite. 
\end{proof}
\begin{proof}[Proof of (GR3)] We consider the maps
\[
\Lambda \to \rho(\Lambda) \hookrightarrow G.
\]
By (GR1) and (GR2) the subset $\rho(\Lambda) \subset G$ is a uniform approximate lattice, hence the second of these maps is a quasi-isometry by Theorem \ref{MilnorSchwarz}. Moreover, by (GR1) the set $\rho(\Lambda)$ is at bounded distance from $\varphi(\Lambda)$. Since $\varphi$ is a quasi-isometry it follows that also the composition of the two maps is a quasi-isometry. We can thus invert this quasi-isometry to obtain a quasi-isometry 
\[
\rho(\Lambda) \hookrightarrow G \to \Lambda,
\]
which is quasi-inverse to $\Lambda \to \rho(\Lambda)$.
\end{proof}
This finishes the proof of Theorem \ref{QIRigidity}.

\subsection{Quasi-isometries from Freiman homomorphisms}\label{SecFreiman}
There is a well-established theory of ``partial homomorphism'' between approximate groups, and we will see that in many cases partial isomorphisms induce quasi-isometries between finitely-generated approximate group. The following definition is taken from \cite{Breuillard}.
\begin{definition} Let $(\Lambda, \Lambda^\infty)$, $(\Xi, \Xi^\infty)$ be approximate groups and $k \in \mathbb N$. A map $\varphi: \Lambda \to \Xi$ is called a 
\emph{$k$-Freiman homomorphism} if for all $g_1, \dots, g_k, \gamma_1, \dots, \gamma_k \in \Lambda$
\[
g_1 \cdots g_k = \gamma_1 \cdots \gamma_k \quad \implies \quad \varphi(g_1) \cdots \varphi(g_k) = \varphi(\gamma_1) \cdots \varphi(\gamma_k).
\]
A \emph{$k$-Freiman isomorphism} is a bijective $k$-Freiman homomorphism whose inverse is also a $k$-Freiman homomorphism.
\end{definition}
By definition, every $k$-Freiman homomorphism extends uniquely to a map $\widehat{\varphi}: \Lambda^k \to \Xi^k$ satisfying $\widehat{\varphi}(gh) = \widehat{\varphi}(g) \widehat{\varphi}(h)$ for all $g \in \Lambda^l$, $h \in \Lambda^m$ as long as $l+m \leq k$. It thus follows that the image of an approximate group under a Freiman $2$-homomorphism is again an approximate group, and that properties of approximate groups which only involve $\Lambda^k$ and partial multiplications $\Lambda^l \times \Lambda^m \to \Lambda^k$ for $l+m \leq k$ are invariant under Freiman $k$-isomorphisms. We will see in Proposition \ref{FreimanConvex} below that under suitable convexity assumptions on the approximate groups involved, Freiman $k$-isomorphisms between finitely-generated approximate groups are quasi-isometries. 

To formulate these assumptions we introduce the following concepts. Given a metric space $X$, we refer to a $(K, C)$-quasi isometric embedding $\varphi: [a,b] \to X$ as a $(K,C)$-quasi-geodesic of length $b-a$. A subset $A \subset X$ will be called \emph{weakly $(R, K, C)$-quasi-convex} in $X$ provided for all $x,y \in A$ there is  a $(K,C)$-quasi-geodesic between $x$ and $y$ contained in $N_R(A)$. If we do not want to specify the parameters we refer to $A$ as a \emph{weakly quasi-convex} subset. The latter notion is invariant under quasi-isometries in the sense that if $A$ is weakly quasi-convex in $X$ and $\varphi: X \to Y$ is a quasi-isometry, then $\varphi(A)$ is weakly quasi-convex in $Y$ (possibly with larger parameters). Indeed, the image of a quasi-geodesic is again a quasi-geodesic, and if $\varphi$ is a $(K,C)$-quasi-isometry, then it maps any subset of $N_R(A)$ to a subset of $N_{KR+C}(\varphi(A))$. 
\begin{definition} A finitely generated approximate group $(\Lambda, \Lambda^\infty)$ is \emph{quasi-convex} if $\Lambda$ is a weakly quasi-convex subset of the Cayley graph ${\rm Cay}(\Lambda^\infty, S)$ for some finite generating set $S \subset \Lambda^\infty$.
\end{definition}
Note that this notion is actually independent of the generating set, since by the previous discussion weak quasi-convexity is a QI-invariant. Assume now that $(\Lambda, \Lambda^\infty)$ is quasi-convex and let $S$ denote a finite generating set of $\Lambda^\infty$. Since for sufficiently large $R = R(S)$ we have
\begin{equation*}\label{NbhdLambda}
N_{k}(\Lambda) \subset N_1(\Lambda^{k+1}) \subset N_{kR}(\Lambda),
\end{equation*}
there exist parameters $(k, R, C)$ depending on $S$ such that for all $x,y \in \Lambda$ there exists an $(R, C)$-quasigeodesic from $x$ to $y$ in ${\rm Cay}(\Lambda^\infty, S)$ whose vertices are contained in $\Lambda^k$. We then say that $(\Lambda, \Lambda^\infty)$ is \emph{$k$-quasi-convex} with respect to $S$ and call $k$ a \emph{convexity parameter} of $(\Lambda, \Lambda^\infty)$ with respect to $S$. 

We will be particularly interested in generating sets $S$ contained in $\Lambda$. If $\Lambda^\infty$ is finitely generated, then such generating sets always exist, since every finite generating set is contained in $\Lambda^k$ for some $k \geq 1$ and every element in $\Lambda^k$ is a finite product of elements in $\Lambda$.
\begin{proposition}\label{FreimanConvex}  Let  $(\Lambda, \Lambda^\infty)$, $(\Xi, \Xi^\infty)$ be quasi-convex finitely-generated approximate groups. Let  $S_\Lambda \subset \Lambda$ and $S_\Xi \subset \Xi$ be finite generating sets with convexity parameter $k$, and let $\varphi: \Lambda \to \Xi$ be a Freiman $(k+1)$-isomorphism with $\varphi(S_\Lambda) = S_\Xi$. Then: 
\begin{enumerate}[(i)]
\item $\varphi$ is a quasi-isometry.
\item $\varphi$ induces an isomorphism of the left-regular quasi-actions of $\Lambda$ and $\Xi$. 
\end{enumerate}
\end{proposition}
\begin{proof} (i) Denote by $\Gamma_k(\Lambda, S_\Lambda)$ the full subgraph of the Cayley graph ${\rm Cay}(\Lambda^\infty, S_\Lambda)$ on the vertex set $\Lambda^k$, and define $\Gamma_k(\Xi, S_\Xi)$ accordingly. Note that in order to construct these graphs we only need to know the partial multiplication $\Lambda^{k} \times \Lambda \to \Lambda^{k+1}$. In particular, every Freiman $(k+1)$-isomorphism $\varphi: \Lambda \to \Xi$ as in (i) induces a graph isomorphism
\[
 \widehat{\varphi}: \Gamma_k(\Lambda, S_\Lambda) \to  \Gamma_k(\Xi, S_\Xi).
\]
By assumption there exist constants $(K,C)$ and for every pair $x,y \in \Lambda$ a $(K, C)$-quasi-geodesic $\gamma_{x,y}$ in ${\rm Cay}(\Lambda^\infty, S_\Lambda)$ between $x$ and $y$ contained in $ \Gamma_k(\Lambda, S_\Lambda)$. The length $\ell(\gamma_{x,y})$ of this curve is then contained in $[\frac{1}{K}d(x,y)-C, K d(x,y) + C]$, and since $\widehat{\varphi}$ is a graph isomorphism, the curve $\widehat{\varphi}(\gamma_{x,y})$ is of the same length. Note that  $\widehat{\varphi}(\gamma_{x,y})$ connects $\varphi(x)$ and $\varphi(y)$ in $\Gamma_k(\Xi, S_\Xi)$, and hence in ${\rm Cay}(\Xi^\infty, S_\Xi)$. This shows that
\[Kd_{S_\Lambda}(x,y) + C \geq \ell(\gamma_{x,y}) = \ell(\widehat{\varphi}(\gamma_{x,y}))  \geq d_{S_{\Xi}}(\varphi(x), \varphi(y)) \quad \text{ for all  }x,y \in \Lambda.\]
and a symmetric argument then yields the existence of constants $K', C'$ such that
\[
(K')^{-1}d_{S_\Lambda}(x,y) - C' \leq d_{S_{\Xi}}(\varphi(x), \varphi(y)) \quad \text{ for all  }x,y \in \Lambda.\]
Since the constants $K,C$ and similarly $K', C'$ are independent of $x$ and $y$, this shows that the bijection $\varphi$ is a quasi-isometry.

(ii) follows from (i), since the induced isomorphism $\varphi_*:{\rm QI}(\Lambda) \to {\rm QI}(\Xi)$ intertwines the corresponding left-regular quasi-actions. 
\end{proof}
This motivates the following problem.
\begin{problem}
Which finitely-generated approximate groups are quasi-convex? 
\end{problem}
It turns out that many, but certainly not all, finitely-generated approximate groups are quasi-convex. Non-examples can be constructed by thickenings of ``almost normal'' non-quasi-convex subgroups as in the following example, which we learned from Emily Stark.
\begin{example}[E. Stark]\label{EStark} Let $\Lambda^\infty := {\rm BS}(1,2) :=  \langle a, b \mid bab^{-1} = a^2 \rangle$ denote the Baumslag-Solitar group of type $(1,2)$, and let $\Lambda :=  \langle a \rangle \cup \{b,b^{-1}\}$. By definition, $\Lambda$ is symmetric, contains the identity and generates $\Lambda^\infty$. A short calculation involving the defining relation (and using that $(b^{-1}ab)^2 = a$) shows that 
\[
\Lambda^2 \subset \Lambda \{e, b, b^{-1}, b^{-1}a\},
\]
hence $(\Lambda, \Lambda^\infty)$ is a finitely-generated approximate group. We claim that it is not quasi-convex. To see this, we fix the generating set $S:=\{a^{\pm 1}, b^{\pm 1}\}$ and consider $w_n := a^{2^n}$. Since $w_n = b^nab^{-n}$ we have $\|w_n\|_S \leq 2n+1$. If $\Lambda$ was weakly quasi-convex inside ${\rm Cay}(\Gamma^\infty, S)$ we would thus find $k \geq 1$ such that $e$ and $w_n$ can be joined by a curve of length $O(n)$ in ${\rm Cay}(\Gamma^\infty, S)$ whose vertices are contained in $\Lambda^k$. However, one can show that for $n > 10k+10$ the shortest such curve has vertices 
\[e, b, \dots, b^{k}, b^{k}a, \dots, b^k a^{2^{n-k}}, \dots,  b^k a^{2^{n-k}}b^{-k} = w_n\]
and length $2k + 2^{n-k}$, which is exponential in $n$.
\end{example}
On the other hand, there do exist many quasi-convex examples:
\begin{proposition} If an approximate group is isomorphic to a uniform model set in a lcsc group, then it is quasi-convex.
\end{proposition}
\begin{proof} Let $(G, H, \Gamma)$ be a uniform cut-and-project scheme and $\Lambda$ be an associated model set with window $W_0$. As an abstract approximate group $(\Lambda, \Lambda^\infty)$ is given by $\Lambda := (G \times W_0) \cap \Gamma$  and $\Lambda^\infty := \Gamma$. Since the inclusion $\Gamma \hookrightarrow G \times H$ is a quasi-isometry, we have quasi-isometries of pairs
\[
(\Lambda, \Lambda^\infty) \to (G \times W_0, G\times H) \to (G, G \times H).
\]
Thus the proposition follows from the fact that $G$ is weakly quasi-convex in $G\times H$.
\end{proof}

\section{From uniform to non-uniform approximate lattices}\label{SecNonUniform}

\subsection{The hull of a closed subset}\label{SecHull} Let $G$ be a lscs group and $\Gamma<G$ be a discrete subgroup. An important object in the study of $\Gamma$ is the homogeneous space $G/\Gamma$. For instance, $\Gamma$ is a uniform lattice if and only if $G/\Gamma$ is compact, and it is a lattice if and only of $G/\Gamma$ admits a $G$-invariant probability measure. If $\Lambda \subset G$ is merely an approximate subgroup, then one can still associate with $\Lambda$ a canonical $G$-space $X_\Lambda$ called the \emph{hull} of $\Lambda$. This $G$-space is typically non-homogeneous, but it can sometimes serve as a weak substitute for the homogeneous space $G/\Gamma$. We now turn to the definition of this space.

Given a lcsc space $X$ we denote by $\mathcal C(X)$ the collection of closed subsets of $X$ with the Chabauty-Fell topology, i.e.\ the topology on $\mathcal C(X)$ generated by the basic open sets 
\[
U_K = \{A \in \mathcal C(X) \mid A \cap K = \emptyset\} \quad \text{and} \quad U^{V} = \{A \in \mathcal C(X) \mid A \cap V \neq \emptyset\},
\]
where $K$ runs over all compact subsets of $X$ and $V$ runs over all open subsets of $X$. Under the present assumptions on $X$, the space $\mathcal C(X)$ is a compact metrizable space (see e.g. \cite[Prop. 1.7 and Prop. 1.8]{Paulin}). The Chabauty-Fell topology has the following convenient property:
\begin{lemma}\label{CFSequences} Assume that $P_n \to P$ in $\mathcal C(X)$. Then for every $p \in P$ there exist elements $p_n \in P_n$ such that $p_n \to p$.
\end{lemma}
\begin{proof} If $p \in P$ then $P \in U^{B_\epsilon(p)}$ for every $\epsilon > 0$. Thus if $P_n \to P$ then for every $\epsilon > 0$ we find $n_0 \in \mathbb N$ such that $P_n \in U^{B_\epsilon(p)}$ for every $n \geq n_0$. Thus there exists a point $p_{n, \epsilon} \in P_n \cap B_\epsilon(p)$ and the lemma follows.
\end{proof}
If $G$ is a lcsc group, then $G\times G$ acts on $\mathcal C(G)$ by
\[
(g,h).\Lambda = g\Lambda h^{-1},
\]
and this action is jointly continuous, since it maps basic open sets to basic open set:
\[
(g,h).U_{K} = U_{gKh^{-1}}, \quad (g,h).U^V = U^{gVh^{-1}} \quad (g,h\in G).
\]
Restricting the action of $G \times G \curvearrowright \mathcal C(G)$ to the factors and the diagonal we obtain three topological dynamical systems over $G$, where $G$ acts from the left, the right or by conjugation. The former two dynamical systems are isomorphic via the isomorphisms $P \mapsto P^{-1}$, but the conjugation system has very different properties. Here we will focus on the action of $G$ on the left as given by $(g, P) \mapsto gP$ for $g \in G$, $P \in \mathcal C(G)$.

\begin{definition} Let $\Lambda \subset G$ be a closed subset. Then the \emph{(right-)hull} $X_\Lambda$ of $\Lambda$ is defined as the closure of the orbit $G.\Lambda$ in $\mathcal C(G)$.
\end{definition}
Note that by definition the hull of a closed subset is always a compact metrizable $G$-space, since it is a closed subset of $\mathcal C(G)$.
\begin{example} If $\Gamma < G$ is a subgroup of $G$, then the map $G/\Gamma \to X_\Gamma$, $[g] \mapsto g\Lambda$ is a continuous injection with dense image. If $\Gamma$ is a uniform lattice, then $G/\Gamma$ is compact, hence we obtain a homeomorphism $G/\Gamma \cong X_\Gamma$. We warn the reader that if $\Gamma< G$ is a non-cocompact subgroup, then the compact space $X_\Gamma$ will always be strictly larger than the image of the non-compact space $G/\Gamma$. In particular, it will always contain the empty-set by the following proposition.
\end{example}
\begin{proposition}\label{Emptyset} A closed subset $\Lambda \subset G$ is relatively dense if and only if $\emptyset \not \in X_\Lambda$.
\end{proposition}
\begin{proof} For every compact $K \subset G$ and open $V \subset G$ we have $\emptyset \in U_K$ and $\emptyset \not \in U^V$, hence the sets $U^K$ generate the neighbourhood filter of $\emptyset$. In particular,  $\emptyset \in X_\Lambda$ if and only if for every $K \subset G$ compact there exists $g_K \in G$ such that $g_K\Lambda  \in U^{K}$, or equivalently $g_K^{-1} \not \in \Lambda K^{-1}$. Thus $\emptyset \not \in X_\Lambda$ if and only if for some compact set $K$ we have $G = \Lambda K^{-1}$, meaning that $\Lambda$ is relatively dense.
\end{proof}
In general, the hull of a discrete set does not need to consist of discrete sets. On the other hand, the hull of a FLC set consists of uniformly discrete sets with uniform parameters:
\begin{proposition}\label{FLCHull} Assume that $\Lambda \subset G$ has finite local complexity. Then there exists an open subset $U\subset G$ such that $|P \cap gU| \leq 1$ for all $P \in X_\Lambda$ and $g \in G$. In particular, for every compact subset $K\subset G$ there exists $C_K>0$ such that for all $P \in X_\Lambda$,
\[
|P \cap K|< C_K.
\]
\end{proposition}
The proof is based on the following lemma.
\begin{lemma}\label{P-1P} For every closed subset $\Lambda \subset G$ and all $P \in X_\Lambda$ we have $P^{-1}P \subset \overline{\Lambda^{-1}\Lambda}$.
\end{lemma}
\begin{proof} If $P \in X_\Lambda$ then there exist $g_n \in G$ such that $g_n\Lambda \to P$. By Lemma \ref{CFSequences} we thus find for every $p,q \in P$ sequence $(p_n)$ $(q_n)$ in $\Lambda$ such that $g_np_n \to p$ and $g_nq_n \to q$. By continuity of multiplication and inversion in $G$ we obtain $p_n^{-1}q_n \to p^{-1}q$ and thus $P^{-1}P \subset \overline{\Lambda^{-1}\Lambda}$.
\end{proof}
\begin{proof}[Proof of Proposition \ref{FLCHull}] By assumption, $\Lambda^{-1}\Lambda$  is closed and discrete. The former implies by Lemma \ref{P-1P} that for all $P \in X_\Lambda$ we have $P^{-1}P \subset \Lambda^{-1}\Lambda$ and the latter implies that there exists an open identity neighbourhood $V$ such that $\Lambda^{-1}\Lambda \cap V = \{e\}$. Combining these two observations we obtain $P^{-1}P \cap V = \{e\}$ for all $P \in X_\Lambda$. Now let $U \subset G$ be an open identity neighbourhood with $U^{-1}U \subset V$. Given $g \in G$ we either have $P \cap gU = \emptyset$ or there exist $p \in P$ and $u \in U$ such that $p = gu$, i.e.\ $g = pu^{-1}$. In the latter case we have
\[
P \cap gU = P\cap pu^{-1}U = p(p^{-1}P \cap u^{-1}U) \subset p(P^{-1}P \cap U^{-1}U) \subset p(P^{-1}P \cap V) = \{p\},
\] 
hence $|P \cap gU| \leq 1$ in either case.
\end{proof}
We record the following consequence of Proposition \ref{FLCHull} for later use.
\begin{corollary}\label{ChabautyConvergence} Let $\Lambda \subset G$ be of finite local complexity, $K \subset G$ compact and asssume that $P_n \to P$ in $X_\Lambda$. Then there exist $k, n_0 \in \mathbb N$ and elements $g_1, \dots, g_k, g_1^{(n)}, \dots, g_k^{(n)} \in G$ such that for all $n \geq n_0$ we have
\[
K \cap P = \{g_1, \dots, g_k\}, \quad K \cap P_n = \{g_1^{(n)}, \dots, g_k^{(n)}\} \quad \text{and} \quad g_i^{(n)} \to g_i.
\]

\end{corollary}
\begin{proof} The set $K \cap P$ is finite by Proposition \ref{FLCHull}, say $K \cap P = \{g_1, \dots, g_k\}$. For every $\epsilon>0$ the set $L := K \setminus N_\epsilon(K \cap P)$ is compact, and since $P \in U_L$ we have $P_n \in U_L$ for all sufficiently large $n$ (depending on $\epsilon$). This means that $P_n \cap K \subset N_\epsilon(K \cap P) = \bigcup B_{\epsilon}(g_i)$. If $\epsilon$ is chosen small enough, then it follows from Proposition \ref{FLCHull} that $|P_n \cap B_{\epsilon}(g_i)| \leq 1$ for all sufficiently large $n$. Then the corollary follows from Lemma \ref{CFSequences}.
\end{proof}

\subsection{Tentative definition of non-uniform approximate lattices}
The goal of this subsection is to discuss various possible definitions of the notion of an approximate lattice $\Lambda$ in a lcsc group $G$. We certainly want $\Lambda$ to be a uniformly discrete approximate subgroup (hence of finite locally complexity), so we assume this from now on. Since we think of the hull $X_\Lambda$ as a substitute for the homogeneous space $G/\Gamma$ of a group, we could simply call $\Lambda$ an approximate lattice if there exists an $G$-invariant probability measure on $X_\Lambda$. However, in this naive definition every non-relatively dense uniformly discrete approximate subgroup would be an approximate lattice, which is clearly not desirable. Indeed, by Proposition \ref{Emptyset} we have $\emptyset \in X_\Lambda$ for any such $\Lambda$ and thus the Dirac measure $\delta_\emptyset$ defines a $G$-invariant measure on $X_\Lambda$. In order to obtain a reasonable definition of an approximate lattice, we have to exclude such measures.
\begin{definition}\label{NonTrivialMeasure} Let $\Lambda \subset G$ be a closed subset. A probability measure $\nu$ on $X_\Lambda$ is called \emph{non-trivial} if $\nu(\{\emptyset\}) = 0$.
\end{definition}
We now have the following first tentative definition of an approximate lattice:
\begin{definition}\label{StrongAL} A uniformly discrete approximate subgroup $\Lambda \subset G$ is called a \emph{strong approximate lattice} if its hull $X_\Lambda$ admits a non-trivial $G$-invariant probability measure.
\end{definition}
\begin{example} Every lattice $\Lambda < G$, uniform or non-uniform, is a strong approximate lattice. Indeed, the unique invariant probability measure on $G/\Lambda$ pushes forward to a non-trivial invariant probability measure on $X_\Lambda$ via the canonical map $G/\Lambda \to X_\Lambda$.
\end{example}
\begin{example} 
Non-uniform non-lattice examples of strong approximate lattices arise again from cut-and-project constructions. Indeed, it was established in \cite[Cor. 3.5]{BHP} that every regular model set $\Lambda$ with symmetric window containing the identity is a strong approximate lattice. In fact, the invariant probability measure on $X_\Lambda \setminus \{\emptyset\}$ is unique in this case. A regular model set is a uniform approximate lattice if and only if the underlying lattice is uniform. This shows that there exist non-uniform strong approximate lattices which are not contained in any lattice.
\end{example}
A problem with Definition \ref{StrongAL} is that we are not able to show that every uniform approximate lattice in an arbitrary lcsc group $G$ is a strong approximate lattice. While there is a natural way to construct measures on the hull of a uniform approximate lattices, these measures will a priori only satisfy a weaker invariance property called \emph{stationarity}. To define this property, let us call a probability measure $\mu$ on $G$ \emph{admissible} if it is absolutely continuous with respect to Haar measure and its support generates $G$ as a semigroup. If $\mu$ is such a measure and $Y$ is a compact $G$ space, then a probability measure $\nu$ on $Y$ is called \emph{$\mu$-stationary} if it is a fixpoint for the convolution action of $\mu$, i.e.\ $\mu \ast \nu = \nu$. 
It follows from the  Kakutani fixed point theorem that if $Y$ is a compact $G$-space, then $Y$-admits a $\mu$-stationary probability measure for every admissible probability measure $\mu$ on $G$. In particular, the hull of a closed subset $\Lambda \subset G$ will always admit a $\mu$-stationary probability measure for every admissible $\mu$, and if $\Lambda$ is relatively dense, then this measure will be non-trivial by Proposition \ref{Emptyset}.
\begin{definition} Let $\Lambda \subset G$ be a uniformly discrete approximate subgroup.
\begin{enumerate}
\item $\Lambda$ is called an  \emph{approximate lattice} if its hull $X_\Lambda$ admits a non-trivial $\mu$-stationary probability measure for every admissible probability measure $\mu$ on $G$.
\item $\Lambda$ is called a \emph{weak approximate lattice} if its hull $X_\Lambda$ admits a non-trivial $\mu$-stationary probability measure for some admissible probability measure $\mu$ on $G$.
\end{enumerate}
\end{definition}
 The discussion preceding the definition shows:
\begin{corollary}\label{UALAL} Every uniform approximate lattice is an approximate lattice.\qed
\end{corollary}
\begin{remark}\label{stst}\label{Strongness}
We have the obvious implications 
\[
\text{strong approximate lattice}  \implies \text{approximate lattice} \implies  \text{weak approximate lattice},
\]
and the question suggests itself, whether these implications can be reversed. This is possible in certain cases, but not in general: 
\begin{enumerate}
\item  By a theorem of Kaimanovich--Vershik \cite{Kaimanovich/Vershik} and independently Rosenblatt \cite{Rosenblatt}, if $G$ is amenable then there exists an admissible probability measure on $G$ such that every $\mu$-stationary measure is actually invariant. Thus if $G$ is an amenable lcsc group, then the notions of an approximate lattice and a strong approximate lattice in $G$ coincide.
\item By definition, a group is non-amenable if and only if there exist some compact $G$-space which admits a $\mu$-stationary probability measure for every admissible $\mu$, but no $G$-invariant probability measure. Thus for non-amenable groups there is a priori no reason to expect that an approximate lattice should be strong. However, we do not know any counterexamples. In fact, we do not even know whether every uniform approximate lattice in a non-amenable lcsc group is strong.
\item A proper subclass of the class of amenable groups is given by the class of \emph{Choquet--Deny groups}. Here a lcsc group $G$ is called a Choquet--Deny group if for every admissible probability measure $\mu$ on $G$, every $\mu$-stationary measure is invariant. By definition, all three notions of approximate lattice coincide for such groups. 
Examples of Choquet--Deny groups include abelian and more generally nilpotent groups \cite{Raugi}. 
\item For some groups $G$ one can also reverse the implication of Corollary \ref{UALAL}. For example, we will see in Theorem \ref{NilpotentUniform} below that in a lcsc nilpotent group $G$ the notions of a uniform approximate lattice, strong approximate lattice, approximate lattice and weak approximate lattice all coincide, and the same holds for discrete groups $G$ by Remark \ref{DiscreteUniformWeak}.
\item Even for amenable groups, there exist weak approximate lattices, which are not approximate lattices. We will see an explicit example in Subsection \ref{NonUnimodular} below.
\end{enumerate}
\end{remark}
\begin{problem} Is every (uniform) approximate lattice in a non-amenable lcsc group a strong approximate lattice?
\end{problem}
\begin{example}[A. Fish]\label{Fish} The following simple example, which was pointed out to us by A. Fish, shows that given a strong approximate lattice $\Lambda \subset G$ we can in general not expect the non-trivial invariant measure on $X_\Lambda$ to be unique, even if $G$ is abelian. Indeed, let $G = \R$ and let $\Lambda \subset \R$ be defined as
\[
\Lambda := \left\{\pm \sum_{j=1}^n x_n \mid n \geq 0 \right\},
\]
where the sequence $(x_n)$ is given by
\[
2,3,2,2,3,3,2,2,2,3,3,3,2,2,2,2,3,3,3,3,\dots
\]
By construction, $\Lambda$ is a symmetric subset of $\Z$ containing $0$. Distances between consecutive points in $\Lambda$ are either $2$ or $3$, and there are arbitrary long blocks of consecutive points of distance $2$ (and similarly, distance $3$) in $\Lambda$.  The former property implies that $\Lambda$ is relatively dense in $\Z$, hence a uniform approximate lattice in $\R$, whereas the latter property implies that $2\Z$ and $3\Z$ are contained in the orbit closure $X_\Lambda$. The two $\R$-orbits of $2\Z$ and $3\Z$ in $X_\Lambda$ give rise to two disjoint closed leaves homeomorphic to $S^1$. Each of these leaves then supports an $\R$-invariant probability measure, hence $X_\Lambda$ is not uniquely ergodic. It is obvious how to modify this example in order to obtain hulls of approximate lattices supporting an arbitrary finite number of disjoint probability meaures.
\end{example}

\subsection{Strong approximate lattices are bi-syndetic}

By definition, an approximate lattice is uniform if and only if it is left-syndetic, or equivalently right-syndetic. While strong approximate lattices need not be uniform, they always satisfy a weaker syndeticity property.
\begin{definition} A subset $\Lambda$ of a lcsc group $G$ is \emph{bi-syndetic} if there exist compact subsets $K, L \subset G$ such that $G = K \Lambda L$.
\end{definition}
For abelian groups, bi-syndeticity is of course equivalent to left-syndeticity. The following example illustrates that for subsets of $\SL_2(\bR)$ bi-syndeticity is a very weak notion.
\begin{example}
Let $G :=  \SL_2(\bR)$, $K := \SO_2(\bR)$ and $A^+ := \{a(t) \mid t \geq 0\}$, where
\[
a(t) := 
\left(
\begin{array}{cc}
e^{t/2} & 0 \\
0 & e^{-t/2}
\end{array}
\right) \quad (t \in \R).
\]
Then we have the \emph{Cartan decomposition} $G = KA^{+}K$, and more precisely every $g \in G$ can be written as  $g = k_1 a(t_g) k_2$ for some $k_1, k_2 \in K$ and a \emph{unique} $t_g \geq 0$, called the \emph{Cartan projection} of $g$. Given a subset $\Lambda \subset G$ let us denote by 
\[
\Sigma_\Lambda := \left\{ t_g \mid g \in \Lambda \right\}
\] 
the set of its Cartan projections. Then by the Cartan decomposition, $\Lambda$ is bi-syndetic provided it is \emph{Cartan-syndetic}, i.e.\ if $\Sigma_\Lambda$ is relatively dense in $[0,\infty)$. Let us now analyze what Cartan-syndeticity amounts to. Given 
\[
g  =
\left(
\begin{array}{cc}
a_g & b_g \\
c_g & d_g
\end{array}
\right) \in \SL_2(\bR), 
\]
we set $s_g := \cosh(t_g) = (a_g^2 + b_g^2 + c_g^2 + d_g^2)/2$. Assume now that $\Lambda \subset G$ is countable and order the elements $g_1,g_2,\ldots$ in $\Lambda$ so that 
\[
t_{g_1} \leq t_{g_2} \leq t_{g_3} \leq \ldots.
\]
Then $\Lambda$ is Cartan-syndetic if and only if there exists $C > 0$ such that $t_{g_{n+1}} - t_{g_n} \leq C$ for every $n$. Since
\begin{eqnarray*}
t_{g_{n+1}} - t_{g_n} 
&=& 
\ln\left(s_{g_{n+1}} + \sqrt{s_{g_{n+1}}^2 - 1}\big) - \ln\big(s_{g_n} + \sqrt{s_{g_n}^2 - 1}\right) \\
&= &\ln\left( \frac{s_{g_{n+1}} + \sqrt{s_{g_{n+1}}^2 - 1}}{s_{g_n} + \sqrt{s_{g_n}^2 - 1}} \right)
\leq \ln 2 + \ln \frac{s_{g_{n+1}}}{s_{g_n}},
\end{eqnarray*}
this holds whenever the quotients $s_{g_{n+1}}/s_{g_n}$ are uniformly
bounded from above. The latter condition can be readily verified for many subsets of interest.
\end{example}

We are going to show:
\begin{theorem}\label{bi-syndetic} Every strong approximate lattice in a lcsc group is bi-syndetic.
\end{theorem}
Since in lcsc abelian groups left-, right- and bi-syndeticity coincide and every weak approximate lattice in a lcsc abelian group is strong we deduce in particular:
\begin{corollary}\label{AbelianUniform} Every weak approximate lattice in a lcsc abelian group is uniform. \qed
\end{corollary}
We will extend this result to nilpotent lcsc groups in the Theorem \ref{NilpotentUniform} below. 
\begin{remark}
We should point out that for an actual lattice $\Gamma$ in a lcsc group $G$ the proof of Theorem \ref{bi-syndetic} is very simple: Just choose a compact subset $K \subset G$ and a pre-compact open identity neighbourhood $U \subset G$ such that $m_G(\Gamma K) > {\rm covol}(\Gamma) = m_G(\Gamma U)$. Then every right-translate of $\Gamma U$ in $\Gamma\backslash G$ will meet $\Gamma K$, which implies that $G = U^{-1}\Gamma K$ and thus $G = \overline{U^{-1}}\Gamma K$. The above proof appears explicitly in \cite[Prop. 2.9]{CapraceMonod}, but the authors point out that the idea goes back at least to \cite[Lemma 1.4]{Borel60}, which in turn is based on an even older (and apparently unpublished) result of Selberg.
\end{remark}
Turning to the proof of Theorem \ref{bi-syndetic} in the general case we introduce the following notation. Let $G$ be a lcsc group and $\Lambda \subset G$ a FLC subset. For the moment we do not assume $\Lambda$ to be an approximate lattice. Since $G$ is second countable, it admits a dense countable subset, and hence a dense countable subgroup $\Gamma < G$, and we fix such a subgroup once and for all. We also fix a proper left-invariant metric $d$ on $G$ inducing the given topology and define a family of open subset $U_\epsilon \subset X_\Lambda$ by
\[
U_\epsilon := \{C \in X_\Lambda \mid C \cap B_{\epsilon}(e) \neq \emptyset\}.
\]
\begin{lemma}\label{UepsilonLemma} For every $\epsilon > 0$ we have $\Gamma U_\epsilon = X_\Lambda\setminus \{\emptyset\}$.
\end{lemma}
\begin{proof} Let $\epsilon > 0$. We compute
\begin{eqnarray*}
\Gamma U_\eps 
&=& 
\bigcup_{\gamma \in \Gamma} \big\{ \gamma C  \in X_\Lambda \mid C \cap B_\eps(e) \neq \emptyset \big\} \\
&=&
\bigcup_{\gamma \in \Gamma} \big\{ \gamma C  \in X_\Lambda \mid \gamma C \cap \gamma B_\eps(e) \neq \emptyset \big\} \\
&=&
\bigcup_{\gamma \in \Gamma} \big\{ C' \in X_\Lambda \mid C' \cap \gamma B_\eps(e) \neq \emptyset \big\} \\
&=&
\big\{ C'  \in X_\Lambda \mid C' \cap \Gamma B_\eps(e) \neq \emptyset \big\}.
\end{eqnarray*}
Since $\Gamma$ is dense in $G$ we have $\Gamma B_\epsilon(e) = G$, and the lemma follows.
\end{proof}
In the sequel, given a subset $\Lambda \subset G$ and a subset $A \subset X_\Lambda$ we denote by
\[
A_\Lambda := \{g \in G\mid g \Lambda \in A\}
\]
the set of \emph{left-return times} of $\Lambda$ to A. Note that for $A, B \subset X_\Lambda$ and $g \in G$ we have
\begin{equation}\label{ReturnTimes}
(A \cap B)_\Lambda = A_\Lambda \cap B_\Lambda \quad \text{and} \quad (gA)_\Lambda  = gA_\Lambda.
\end{equation}
\begin{corollary} If $\Lambda \subset G$ is a closed subset such that $X_\Lambda$ admits a non-trivial $G$-invariant probability measure $\nu$, then for every $\epsilon > 0$ there is a finite set $F \subset G$ such that
\[
G = FB_\epsilon(e) \Lambda^{-1}\Lambda B_{\epsilon}(e).
\]
In particular, $ \Lambda^{-1}\Lambda$ is bi-syndetic.
\end{corollary}
\begin{proof} Fix $\epsilon >0$ and set $\delta := \nu(U^\eps)$. By Lemma \ref{UepsilonLemma} we have $\nu(\Gamma U^\eps) = \nu(X_\Lambda \setminus \{\emptyset\}) = 1$, so we can find a finite subset $F\subset \Gamma \subset G$ such that $\nu(FU^\eps) > 1 - \delta$. Now for every $g \in G$ we have $\nu(gU^\eps) = \delta$, hence 
\[
\nu(FU^\eps \cap g U^\eps) > 0.
\]
In particular, for every $g \in G$ the set $FU^\eps \cap g U^\eps \subset X_\Lambda$ is a non-empty open set, and hence meets the $G$-orbit of $\Lambda$ non-trivially, i.e.\ $(FU^\eps \cap gU^\eps)_\Lambda \neq \emptyset$. Now observe that
\[
(U^\eps)_\Lambda  = \big\{ g \in G \mid g\Lambda \cap B_\eps(e) \neq \emptyset \big\} = \big\{ g \in G\mid g \cap B_\eps(e) \Lambda^{-1} \neq \emptyset \big\} = B_\eps(e) \Lambda^{-1}.
\]
We deduce with \eqref{ReturnTimes} that \[
\emptyset \neq (FU^\eps \cap gU^\eps)_\Lambda = F(U^\eps)_\Lambda \cap g (U^\eps)_\Lambda
= FB_\eps(e) \Lambda^{-1} \cap g B_\eps(e) \Lambda^{-1},
\]
for every $g \in G$, and thus
\[
G = FB_\eps(e) \Lambda^{-1} \Lambda B_\eps(e)^{-1},
\]
which finishes the proof.
\end{proof}
\begin{proof}[Proof of Theorem \ref{bi-syndetic}] If $\Lambda \subset G$ is a strong approximate subgroup, then $\Lambda^{-1}\Lambda = \Lambda^2 \subset F_o\Lambda$ for some finite $F_o$, and thus $G = FB_\epsilon(e) \Lambda^{-1}\Lambda B_{\epsilon}(e) = FB_\epsilon(e)F_o \Lambda B_{\epsilon}(e)$.
\end{proof}

If $G$ happens to be discrete, then we may assume that $d$ takes only integral values and thus $B_\epsilon(e) = \{e\}$ for $\epsilon < 1$. We thus recover a classical result of F\o lner \cite{Folner}:
\begin{corollary}\label{DiscreteUniform} If $G$ is discrete and $\Lambda \subset G$ is a subset such that $X_\Lambda$ admits a non-trivial $G$-invariant probability measure $\nu$, then $\Lambda^{-1}\Lambda$ is right-syndetic. In particular, every strong approximate lattice in a discrete group is uniform. \qed
\end{corollary}
\begin{remark}\label{DiscreteUniformWeak}
In fact, it follows from results in \cite{Bjo} that every \emph{weak} approximate lattice in a discrete group is uniform as well. Indeed, if $G$ is a discrete group, $\mu$ an admissible probability measure on $G$ and $\Lambda \subset G$ is a subset whose right hull admits a non-trivial $\mu$-stationary
probability measure, then by  \cite[Theorem 1.9]{Bjo} the difference set $\Lambda^{-1} \Lambda$ equals
the intersection of a right syndetic set and a ``left thick'' set, whence 
$(\Lambda^{-1} \Lambda)^2$ is right syndetic (note that $\Lambda^{-1}$ is a $\cL_\mu$-large
set in the notation of \cite{Bjo}). In particular, if $\Lambda$ is symmetric,
then $\Lambda^4$ is right syndetic. 
Hence, if we in addition assume that $\Lambda$ is an 
approximate group (so that it is a weak approximate lattice in $G$), then we conclude from 
above that $\Lambda$ must be right syndetic in $G$, and  thus a uniform approximate lattice in $G$. 
\end{remark}

\subsection{Approximate lattices in nilpotent groups}
We have seen in the previous subsection that every approximate lattice in a lcsc abelian group is uniform. In this subsection we extend this result to nilpotent lcsc groups:
\begin{theorem}\label{NilpotentUniform} Every weak approximate lattice in a nilpotent lcsc group is uniform.
\end{theorem}
\begin{remark}\label{padic} Concerning Theorem \ref{NilpotentUniform}, it is instructive to compare the class of nilpotent Lie groups to the class of semisimple $p$-adic groups. Every lattice is uniform in groups from either class, but for very different reasons. In the case of nilpotent Lie groups the reason for the non-existence of non-uniform lattices is geometric, and we explain below how to extend the argument to show non-existence of non-uniform approximate lattice. In the p-adic case, the reason for the non-existence of non-uniform lattice is purely group-theoretic. Consider for example the groups $G_p := {\rm SL}_2(\Q_p)$ and denote by $V_p$ the set of vertices of their respective Bruhat--Tits trees. Then a discrete subgroup $\Gamma  < G_p$ is a lattice if and only if
\[
\sum_{[x] \in \Gamma\backslash V_p} \frac{1}{|\Gamma_x|} < \infty.
\]
Since $\Gamma_x < G_p$ is finite for every $x \in V_p$ and the size of finite subgroups of $G_p$ is uniformly bounded, this is possible only if $|\Gamma\backslash V_p| < \infty$, i.e.\@ if $\Gamma$ is uniform. Thus the reason for the non-existence of non-uniform lattices in ${\rm SL}_2(\Q_p)$ is that the latter does not contain torsion subgroups of arbitrary large order.
No such obstruction exists in the approximate setting, and in fact the groups ${\rm SL}_2(\Q_p)$ (and more generally, the groups ${\rm SL}_n(\Q_p)$ for $n \geq 2$) do admit non-uniform approximate lattice. Explicit examples are given by regular model sets arising from the non-uniform cut and project scheme $({\rm SL}_n(\Q_p), {\rm SL}_n(\R), {\rm SL}_n(\Z[1/p]))$. In particular, the analogue of Theorem \ref{NilpotentUniform} does not hold for the class of semisimple $p$-adic groups.
\end{remark}
For the proof of Theorem \ref{NilpotentUniform} we introduce the following terminology:
\begin{definition}
A lcsc group $G$ is called \emph{balanced} if every bi-syndetic approximate subgroup of $G$ is left- (equivalently, right-) syndetic. 
\end{definition}
By Theorem \ref{bi-syndetic} every strong approximate lattice in a balanced lcsc group is uniform. Since every weak approximate lattice in a nilpotent group is strong, the proof of Theorem \ref{NilpotentUniform} reduces to showing that every nilpotent lcsc group is balanced. We are going to show this by induction on the nilpotency degree. Obviously every $1$-step nilpotent, i.e.\ abelian lcsc group $G$ is balanced. The induction step then amounts to proving the following proposition, which is also of independent interest.
\begin{proposition}\label{NilpotentMainProp} Let $0 \to Z \to G \xrightarrow{\pi} Q \to \{e\}$ be a central extension of lcsc groups. If $Q$ is balanced, then so is $G$.
\end{proposition}
The proof of the proposition will occupy the remainder of this subsection. We fix a central extension $0 \to Z \to G \xrightarrow{\pi} Q \to \{e\}$ and a Borel section $s: Q \to G$ of $\pi$, which we assume to be symmetric (i.e.\ $s(q)^{-1} = s(q^{-1})$ and locally bounded (i.e.\ images of compact sets are pre-compact). Given $\Omega \subset G$ and $q \in Q$ we then denote by
\[
\Omega_q := \{z \in Z \mid zs(q) \in \Omega\} = \{z \in Z \mid s(q)z \in \Omega\} \subset Z
\]
the ``fiber'' over $q$ so that
\begin{equation}\label{FiberUnion}
\Omega = \bigcup_{q \in \pi(\Omega)} \Omega_q s(q).
\end{equation}
By definition we have for all $\Omega_1, \Omega_2 \subset G$ and $q \in Q$ the inclusion
\begin{equation}\label{ProductFibers}
(\Omega_1\Omega_2)_q \supset (\Omega_1)_q (\Omega_2)_e.
\end{equation}
We now assume that $Q$ is balanced and fix a bi-syndetic approximate subgroup $\Lambda \subset G$. Then $\pi(\Lambda)$ is a bi-syndetic approximate subgroup of $Q$, hence left-syndetic by assumption, say $Q = E\pi(\Lambda)$ for some $E \subset Q$ compact. 
\begin{lemma}\label{MainLemmaNilpotent} Let $\Lambda, E$ be as above and assume that there exists a pre-compact set $N \subset Z$ such that for all $q \in \pi(\Lambda)$ we have $(N\Lambda^2)_q = Z$. Then
$G = s(E)N\Lambda^2$. In particular, $\Lambda$ is right-syndetic in $G$.
\end{lemma}
\begin{proof} Since $e \in \Lambda$ and $N \subset Z = \ker(\pi)$ we have $\pi(\Lambda) \subset \pi(\Lambda^2) = \pi(N\Lambda^2)$. We thus deduce from \eqref{FiberUnion} that
\[
N\Lambda^2 = \bigcup_{q \in \pi(N\Lambda^2)} (N\Lambda^2)_q s(q) =  \bigcup_{q \in \pi(\Lambda^2)} (N\Lambda^2)_q s(q) \supset \bigcup_{q \in \pi(\Lambda)} (N\Lambda^2)_qs(q) = Zs(\pi(\Lambda)),
\]
and hene
\[
s(E)N\Lambda^2 \supset s(E) Z s(\pi(\Lambda)) = Z s(E) s(\pi(\Lambda)) = Z s(E\pi(\Lambda)) = Zs(Q) = G,
\]
where we have used the assumptions that $E \pi(\Lambda) = Q$ and $Zs(Q) = G$ together with the observation that since $\ker \pi = Z$, we have $Zs(A)s(B) = Zs(AB)$ for all subsets $A, B \subset Q$.
\end{proof}
We are thus left with the task to construct for any given $\Lambda$ a pre-compact set $N \subset Z$ as in Lemma \ref{MainLemmaNilpotent}. By assumption we have $G = K\Lambda L$ for compact sets $K, L \subset G$.  We may assume that $K = K_Zs(K_Q)$, $L = L_Zs(L_Q)$ for compact sets $K_Z, L_Z \subset Z$ and $K_Q, L_Q \subset Q$. We then define 
\[
M_o := \big\{ s(k_Q) s(k_Q^{-1} l_Q^{-1})s(l_Q) \mid k_Q \in K_Q, \enskip l_Q \in L_Q \big\} \subset Z \quad \text{and} \quad M:= K_ZL_ZM_o \subset Z.
\]
We also set $\Sigma := \pi(\Lambda) \cap K_QL_Q$ and
\begin{equation}\label{DefNNilpotent}
N_o := s((K_QL_Q)^{-1}) \subset G \quad \text{and} \quad N:= MN_o \subset G.
\end{equation}
We will show that $N$ satisfies the assumptions of Lemma \ref{MainLemmaNilpotent} and thereby finish the proof. Note that $N$ is pre-compact by construction. To show that $(N\Lambda^2)_q = Z$ for all $q \in \pi(\Lambda)$ we need two lemmas.
\begin{lemma}\label{efiber}  $(N_o\Lambda)_e = \bigcup_{q\in \Sigma}\Lambda_q$.
\end{lemma}
\begin{proof} By \eqref{FiberUnion} we have
\[
\Lambda = \bigcup_{u \in \pi(\Lambda)} \Lambda_u s(u),
\]
whence
\[
N_o\Lambda = (s((K_QL_Q)^{-1})\Lambda = \bigcup_{t \in K_Q L_Q} \bigcup_{u \in \pi(\Lambda)} s(t^{-1})s(u) \Lambda_u.
\]
We see that the only sets in the union which contribute to the fiber above $e$ have $t = u = :q$, and thus 
\[
\big(N_o\Lambda\big)_e = \bigcup_{q \in K_Q L_Q \cap \pi(\Lambda)} \Lambda_q,
\]
which finishes the proof.
\end{proof}
\begin{lemma}
\label{KL}
For every $z \in Z$, there exists $q \in \Sigma$ such that $z \in M  \Lambda_q$. In particular, 
\[
M\big(\bigcup_{q \in \Sigma} \Lambda_q\big) = Z.
\]
\end{lemma}

\begin{proof} Every $z \in Z \subset G = K\Lambda L$ can be written as 
\[
z  = k_Z s(k_Q) \lambda l_Z s(l_Q)  =  k_Z s(k_Q)\lambda_Z
s(\pi(\lambda))  l_Z s(l_Q)
\]
with $k_Z \in K_Z$, $k_Q \in K_Q$, $\lambda \in \Lambda$, $\lambda_Z \in \Lambda_{\pi(\lambda)}$, $l_Z \in L_Z$ and $l_Q \in L_Q$. Note that
\[
e = \pi(z) = k_Q\pi(\lambda)l_Q
\]
and thus  
\[q := \pi(\lambda) = k_Q^{-1}l_Q^{-1} \in K_Q L_Q \cap \pi(\Lambda) = \Sigma.\]
Since $\lambda_Z$, $k_Z$ and $l_Z$ are central we conclude that
\begin{eqnarray*}
z 
&=& 
k_Z s(k_Q) \lambda_Z s(\pi(\lambda)) l_Z s(l_Q) \\
&=& 
 k_Z \ l_Z s(k_Q) s(\pi(\lambda)) s(l_Q)\lambda_Z \\
&=&
 k_Z l_Z s(k_Q) s(k_Q^{-1}l_Q^{-1}) s(l_Q)\lambda_Z\\ &\in&  K_ZL_ZM_o\lambda_Z=M\lambda_Z,
\end{eqnarray*}
and since $\lambda_Z \in \Lambda_{\pi(z)} = \Lambda_q$ and $q \in \Sigma$ the lemma follows.
\end{proof}
\begin{proof}[Proof of Proposition \ref{NilpotentMainProp}] In view of Lemma \ref{MainLemmaNilpotent} it remains to show only that the set $N$ defined in \eqref{DefNNilpotent} satisfies
$(N\Lambda^2)_q = Z$ for all $q \in \pi(\Lambda)$. Since $M$ is contained in $Z$, Lemma \ref{efiber} and Lemma \ref{KL} yield
\[
(N\Lambda)_e = (MN_o\Lambda)_e = M (N_o\Lambda)_e = M\left(\bigcup_{p\in \Sigma}\Lambda_p\right) = Z.
\]
Now let $q \in \pi(\Lambda)$. We apply apply \eqref{ProductFibers} with $\Omega_1 = N\Lambda$ and $\Omega_2 = \Lambda$ to obtain
\[
(N\Lambda^2)_q \supset (N\Lambda)_e \Lambda_q  = Z \Lambda_q.
\]
Since $\Lambda_q \neq \emptyset$ we deduce that $(N\Lambda^2)_q = Z$, finishing the proof.
\end{proof}
This concludes the proof of Theorem \ref{NilpotentUniform}.

\section{Unimodularity}\label{SecUnimodular}

\subsection{The periodization map}

If $\Gamma < G$ is a lattice in a lcsc group, then there is a periodization map
\[
\mathcal P_\Gamma: C_c(G) \to C_c(G/\Gamma), \quad \mathcal P_\Gamma(f)(g\Gamma) := \sum_{\gamma \in \Gamma} f(g\gamma),
\]
and we are going to define an analogous periodization map for approximate lattices. Recall from Proposition \ref{FLCHull} that if $\Lambda \subset G$ is a uniformly discrete approximate subgroup of a lcsc group $G$, then every $P \in X_\Lambda$ is locally finite. More precisely, for every compact set there exists a constant $C_K$ such that 
\begin{equation}\label{PeriodBound}|P \cap K| < C_K\quad \text{for all } P \in X_\Lambda.
\end{equation}
In particular, given $f \in C_c(G)$, we can define the \emph{periodization} $\mathcal Pf$ of $f$ along $\Lambda$ by the finite sums
\[
\mathcal Pf(P) := \sum_{x \in P}f(x) \quad (P \in X_\Lambda)
\]
Note that the map $f \mapsto \mathcal Pf$ is equivariant with respect to the left-action of $G$ on itself and the $G$-action on $X_\Lambda$.
\begin{proposition} It $\Lambda \subset G$ is an approximate lattice, then for every $f \in C_c(G)$ the periodization $\mathcal Pf : X_\Lambda \to \R$ is continuous with respect to the Chabauty-Fell topology.
\end{proposition}
\begin{proof} Let $K:= {\rm supp}(f)$ and assume $P_n \to P$ in $X_\Lambda$. By Corollary \ref{ChabautyConvergence} we have $P \cap K = \{g_1, \dots, g_k\}$ and $P_n \cap K = \{g_1^{(n)}, \dots, g_k^{(n)}\}$ with $g_i^{(k)} \to g_i$ for all sufficiently large $n$. We deduce that
\[
\mathcal Pf(P_n) = \sum_{x \in P_n \cap K} f(x) = \sum_{i=1}^k f(g_i^{(n)}) \to \sum_{i=1}^k f(g_i) = \sum_{g \in P \cap K} f(x)  = \mathcal Pf(P). \qedhere
\]
\end{proof}
\begin{definition} The map $\mathcal P: C_c(G) \to C(X_\Lambda)$, $f \mapsto \mathcal P f$ is called the \emph{periodization map} of the uniformly discrete approximate subgroup $\Lambda \subset G$.
\end{definition}
One important difference to the group case concerns the range of the periodization map: If $\Gamma < G$ is a uniform lattice, then the periodization map $C_c(G) \to C(G/\Gamma)$ is in fact surjective (see e.g. \cite[Lemma 1.1]{Raghunathan}). This need not be the case for uniform approximate lattices. However, we at least have:
\begin{proposition}\label{PointSeparation} The image $\mathcal P(C_c(G))$ of the periodization map separates points in $X_\Lambda \setminus \{\emptyset\}$.
\end{proposition}
\begin{proof} Let $P_1, P_2 \in X_\Lambda$ be distinct points. Changing enumeration if necessary we may assume that there exists $x \in P_1 \setminus P_2$. Since $P_2$ is locally finite by Proposition \ref{FLCHull}, there exists $r>0$ such that $B_r(x) \cap P_2 = \emptyset$. Now choose $f \in C_c(G)$ with $f \geq 0$, $f(x) >0$ and ${\rm supp}(f) \subset B_r(x)$. Then $\mathcal Pf(P_1) \geq f(x)  >0$ and $\mathcal Pf(P_2) = 0$, hence $\mathcal Pf$ separates $P_1$ and $P_2$. 
\end{proof}

\subsection{Periodization of measures}
We can use the periodization map to transfer measures on the hull to measures on the group. To make this precise, we recall that a \emph{Radon measure} on $G$ is a positive linear functional $\eta: C_c(G) \to \R$ such that for every compact subset $K \subset G$  there is a constant $C_K$ such that for every $f \in C_c(G)$ with ${\rm supp}(f) \subset K$ we have
\[
\eta(f) \leq C_K \cdot \|f\|_\infty.
\]
Assume now that we are given a probability measure $\nu$ on $X_\Lambda$. We then obtain a linear functional $\eta:= \mathcal P^*\nu$ on $C_c(G)$ by setting
$\eta(f) := \nu(\mathcal Pf)$. It turns out that $\eta$ is a Radon measure by the following proposition.
\begin{proposition} For every compact subset $K \subset G$ there exists a constant $C_K$ such that for $f \in C_c(G)$ with ${\rm supp}(f) \subset K$ we have 
\[
\|\mathcal Pf\|_\infty \leq C_K \cdot \|f\|_\infty.
\]
\end{proposition}
\begin{proof} If we choose $C_K$ as in \eqref{PeriodBound}, then 
\[
|\mathcal Pf(P)| \leq \sum_{x \in P \cap K} |f(x)| \leq |P \cap K| \cdot \|f\|_\infty \leq C_K \cdot  \|f\|_\infty.\qedhere
\]
\end{proof}
\begin{definition} The Radon measure $\mathcal P^*\nu$ is called the periodization of the probability measure $\nu$ on $X_\Lambda$.
\end{definition}
\begin{lemma} If $\nu$ is a non-trivial measure on $X_\Lambda$ in the sense of Definition \ref{NonTrivialMeasure}, then $\mathcal P^*\nu$ is non-zero.
 \end{lemma}
\begin{proof} Choose $P \in {\rm supp}(\nu) \setminus \{\emptyset\}$; by Proposition \ref{PointSeparation} there exists $f \in C_c(G)$ satisfying $f \geq 0$ and  $\mathcal Pf(P)  >0$,  hence there is $\epsilon$ such that $\mathcal Pf \geq \epsilon$ on an open subset of ${\rm supp}(\nu)$. We deduce that $\mathcal P^*\nu(f) = \nu(\mathcal Pf) > 0$.
\end{proof}
Since the periodization map is $G$-equivariant, the periodization of a $G$-invariant measure on $X_\Lambda$ is invariant under the $G$-action on itself by left-multiplication. Similarly, periodization preserves stationarity, but some care has to be taken to make this statement precise. Namely, given a Radon measure $\eta$ on $G$ and an admissible probability measure $\mu$ on $G$ the convolution $\mu \ast \eta$ may not be defined, since the integral may not converge. This problem does not occur if $\mu$ is compactly supported. In this case, we call $\eta$ a \emph{$\mu$-stationary Radon measure} provided $\mu \ast \eta = \eta$. With this terminology understood we have:
\begin{corollary}\label{PeriodizationCts}  Let $\mu$ be a compactly-supported admissible probability measure on $G$. If $\nu$ is a $\mu$-stationary probability measure on $X_\Lambda$, then $\mathcal P^*\nu$ is a $\mu$-stationary Radon measure on $G$. If $\nu$ is $G$-invariant, then so is $\mathcal P^*\nu$, and if $\nu$ is non-trivial, then $\mathcal P^*\nu$ is non-zero.\qed
\end{corollary}

\subsection{The unimodularity theorem}

Recall that a lcsc group $G$ is called unimodular if every left-Haar measure $m_G$ on $G$ is a right-Haar measure. Examples of unimodular groups include all discrete, compact and simple lcsc groups and their products. If a lcsc group $G$ contains a lattice, then it must be unimodular. Here we establish the following generalization of this result:
\begin{theorem}\label{ThmUnimodular} Let $G$ be a lcsc group. Assume that either
\begin{enumerate}
\item $G$ contains a strong approximate lattice $\Lambda$; or
\item $G$ contains a uniform approximate lattice $\Lambda$ and is compactly generated.
\end{enumerate}
Then $G$ is unimodular.
\end{theorem}
\begin{problem} Let $G$ be a non-amenable compactly generated lcsc group which contains an approximate lattice. Is $G$ necessarily unimodular?
\end{problem}
For the proof of Theorem \ref{ThmUnimodular} we denote by $\Delta_G: G \to \R^{>0}$ the modular function of $G$. We use the convention that
$m_G(Ag) = \Delta_G(g) m_G(A)$ for any pre-compact measurable set $A \subset G$ or equivalently 
\[
\int_G f(xg)\,dm_G(x) = \Delta_G(g)^{-1} \int_G f(x) \,dm_G(x) \quad (f \in C_c(G), g \in G).
\]
We warn the reader that the opposite convention is also in use. Note that unimodularity of $G$ amounts to $\Delta_G \equiv 1$.

The proof of Theorem \ref{ThmUnimodular} will make use of the periodization map $\mathcal P: C_c(G) \to C(X_\Lambda)$. Note that if $\Gamma< G$ is a lattice, then 
the periodization map $\mathcal P: C_c(G) \to C_c(G/\Gamma)$ is not only equivariant with respect to the left-regular action of $G$, but also invariant under the action of $\Gamma$ on $G$ by multiplication on the right. For approximate uniform lattices, this invariance still holds approximately. Indeed, given $f \in C_c(G)$ and $t \in \Lambda$, denote by $f \cdot t$ the function $g\mapsto f(gt)$; then we have:
\begin{lemma}\label{LemmaUnimodularityMainEstimate} Assume that $\Lambda \subset G$ is a uniform approximate lattice and $F \subset G$ finite with $\Lambda^2 \subset \Lambda F$. Then for every $t \in \Lambda$ and $f \in C_c(G)$ with $f \geq 0$ we have
\begin{equation*}\label{UnimodularityMainEstimate}
\mathcal P(f \cdot t) \leq \sum_{c \in F} \mathcal P(f \cdot c).
\end{equation*}
\end{lemma}
\begin{proof} For all $g \in G$ and $t \in \Lambda$ we have
\[
\mathcal P(f \cdot t)(g\Lambda) = \sum_{\lambda \in \Lambda}f(g\lambda t) \leq \sum_{\lambda \in \Lambda^2}f(g\lambda) \leq \sum_{\lambda \in \Lambda F}f(g\lambda) \leq \sum_{c \in F}\mathcal P(f \cdot c)(g\Lambda),
\]
and since $G.\Lambda \subset X_\Lambda$ is dense the lemma follows.
\end{proof}
Part (1) of the theorem now follows by combining this lemma with Theorem \ref{bi-syndetic}:

\begin{proof}[Proof of Theorem \ref{ThmUnimodular}(1)] Let $\nu$ be a non-trivial $G$-invariant probability measure on $X_\Lambda$ and denote by $\eta := \mathcal P^*\mu$ its periodization. By Corollary \ref{PeriodizationCts}, $\eta$ is non-zero left-$G$-invariant Radon measure on $G$, i.e.\ a left-Haar measure.

By Lemma \ref{LemmaUnimodularityMainEstimate} we have for every $f\in C_c(G)$ with $f \geq 0$ and every $t \in \Lambda$,
\[
\eta(f \cdot t) = \nu\left(\mathcal P(f \cdot t)\right) \leq \sum_{c \in F} \nu\left(\mathcal P(f \cdot c)\right) = \sum_{c \in F} \eta(f \cdot c).
\]
Since $\eta$ is a left-Haar measure we obtain
\[
\Delta_G(t^{-1})\eta(f) = \eta(f \cdot t) \leq \sum_{c \in F} \eta(f \cdot c) = \left(\sum_{c\in F}\Delta_G(c^{-1})\right)\eta(f),
\]
and thus for all $t \in \Lambda$,
\[
\Delta_G(t^{-1}) \leq \sum_{c \in F} \Delta_G(c^{-1}).
\]
Applying this inequality to both $t$ and $t^{-1}$ we deduce that the homomorphism $\log \Delta_G: G \to \R$ is bounded uniformly on $\Lambda$.

On the other hand, by Theorem \ref{bi-syndetic} there exist compact subsets $K, L \subset G$ such that $G =  K\Lambda L$. Since $\log \Delta_G$ is continuous, it is bounded on the compact sets $K$ and $L$. Since it is moreover a homomorphism and bounded on $\Lambda$, it is thus bounded on all of $G$. Since $\R$ has no non-trivial bounded subgroups, we deduce that $\log \Delta_G \equiv 0$, i.e., that $G$ is unimodular.
\end{proof}
To establish unimodularity also for uniform approximate lattices, which are not strong, we need to work with stationary measures instead of invariant measures. The main new ingredient is the construction of a specific compactly supported admissible probability measure $\mu$ on $G$ with a continuous density with the following special properties.
\begin{lemma}\label{UnimodularityFunction} Assume that $G$ is a compactly generated non-unimodular lcsc group. Then there exists $\rho \in C_c(G)$ with the following properties.
\begin{enumerate}[(i)]
\item $\rho \geq 0$ and $\int_G \rho(s) dm_G(s) = 1$.
\item ${\rm supp}(\rho)$ generates $G$ as a semigroup.
\item $\int_G \rho(s) \Delta_G(s) dm_G(s) < 1$.
\end{enumerate}
\end{lemma}
\begin{proof} Clearly there exists $\rho_o \in C_c(G)$ satisfying (1) and (2). Define
\[
\gamma := \int_G \rho_o(s) \Delta_G(s) dm_G(s).
\]
Let $a >0$ such that $a\gamma < \frac 1 2$. Since the homomorphism $\Delta$ is unbounded, we can find $s \in G$ such that $(1-a)\gamma\Delta(s)^{-1}< 1/2$. Fix such an $s \in G$ and define
\[
\rho(t) := a\rho_o(t) + (1-a) \rho_o(st).
\]
Then $\rho$ still satisfies (1) (by left-invariance of $m_G$) and (2), and we have
\begin{eqnarray*}
\int_G \rho(t) \Delta(t) dm_G(t) &=& a \int_G \rho_o(t) \Delta_G(t) dm_G(t) + (1-a) \int_G \rho_o(st) \Delta_G(t) dm_G(t)\\
&=& a\gamma + (1-a) \int_G \rho_o(t) \Delta_G(s^{-1}t) dm_G(t)\\
&=&  a\gamma + (1-a)\gamma\Delta(s)^{-1},
\end{eqnarray*}
which by assumption is strictly smaller than $1/2+1/2 = 1$.
\end{proof}
Combining this construction with the estimate in \eqref{UnimodularityMainEstimate} we can now finish the proof of Theorem \ref{ThmUnimodular}.
\begin{proof}[Proof of Theorem \ref{ThmUnimodular}(2)] Let $G$ be a non-discrete compactly generated lcsc group and $\Lambda \subset G$ a uniform approximate lattice. We assume for contradiction that $G$ is not unimodular and define a compactly-supported admissible probability measure $\mu$ on $G$ by $\mu := \rho m_G$ on $G$, where $\rho$ is chosen as in Lemma \ref{UnimodularityFunction}. We denote by $\nu$ a non-trivial $\mu$-stationary probability measure on $X_\Lambda$ and define $\eta := \mathcal P^*\nu$. By Corollary \ref{PeriodizationCts}, $\eta$ is a non-zero $\mu$-stationary Radon measure on $G$. Stationarity implies that $\eta$ has a continuous density $u: G \to \R^{>0}$ which is $\rho$-harmonic, i.e, for all $t \in G$,
\[
u(t) = (\rho \ast u)(t) = \int_G \rho(s) u(s^{-1}t)\; dm_G(s).
\]
By  \eqref{UnimodularityMainEstimate} we have, for every $f\in C_c(G)$ with $f \geq 0$ and every $t \in \Lambda$,
\[
\eta(f \cdot t) = \nu\left(\mathcal P(f \cdot t)\right) \leq \sum_{c \in F} \nu\left(\mathcal P(f \cdot c)\right) = \sum_{c \in F} \eta(f \cdot c).
\]
Since $\eta = u m_G$ we have for all $g \in G$ and $t \in \Lambda$,
\[
\eta(f \cdot t) =  \int_G f(gt)u(g) dm_G(g) = \int_G f(g) u(gt^{-1})\Delta_G(t)^{-1} dm_G(g),
\]
which allows us to rewrite the previous inequality as
\[
\int_G f(g) \left(u(gt^{-1})\Delta_G(t)^{-1}\right)dm_G(g) \leq \int_G f(g)  \left( \sum_{c \in F}u(gc^{-1})\Delta_G(c)^{-1}\right)  dm_G(g).
\]
Since this holds for every $f$ and $u$ and $\Delta_G$ are continuous, we obtain for all $g \in G$ and $t \in \Lambda$,
\begin{equation}\label{ModularEstimate}
u(gt^{-1})\Delta_G(t)^{-1} \leq \sum_{c \in F}u(gc^{-1})\Delta_G(c)^{-1}.
\end{equation}
Now let $K \subset G$ be a compact subset satisfying $G = K\Lambda$. Every $g \in G$ can then be written as $g = kt$ with $k \in K$ and $t \in \Lambda$, and hence by \eqref{ModularEstimate} we obtain
\[
u(g)\Delta_G(g) = u(kt^{-1})\Delta_G(kt^{-1}) = \Delta_G(k) \cdot u(kt^{-1})\Delta_G(t^{-1}) \leq  \Delta_G(k) \cdot \sum_{c \in F}u(kc^{-1})\Delta_G(c)^{-1}. 
\]
Since $K$ and $F$ are compact and $u$ and $\Delta_G$ are continuous we thus find a uniform constant $M$ such that for all $g \in G$
\[
u(g)\leq M \cdot \Delta_G(g)^{-1}.
\]
Since $u$ is $\rho$-harmonic it follows that for every $n > 0$,
\[
u(e) = (\rho^{\ast n} \ast u)(e) \leq M( \rho^{\ast n} \ast \Delta_G^{-1})(e) = M \cdot  \int_G \rho^{\ast n}(s) \Delta_G(s) dm_G(s).
\]
Since $\Delta_G$ is a homomorphism and $\int \rho \Delta_G dm_G < 1$ we deduce that
\begin{eqnarray*}
u(e) &\leq&  \int_G\cdots \int_G \rho(s_1^{-1}s_2) \rho(s_2^{-1}s_3) \dots \rho(s_{n-1}^{-1}s_n)\Delta_G(s_1 \cdots s_n) \; dm_G(s_1) \dots dm_G(s_n)\\
&=&\left( \int_G \rho(s)\Delta_G(s)\right)^n \quad \overset{n \to \infty}\longrightarrow \quad 0,
\end{eqnarray*}
i.e.\ $u(e) = 0$. This implies that for all $n >0$
\[
0 = u(e) =  (\rho^{\ast n} \ast u)(e) = \int_G \rho^{\ast n}(s) u(s^{-1})dm_G(s),
\]
and since ${\rm supp}(\rho)$ generates $G$ as a semigroup and $u$ is continuous we conclude that $u \equiv 0$ and thus $\eta = 0$, which is a contradiction.
\end{proof}
\subsection{A weak approximate lattice in a non-unimodular lcsc group}\label{NonUnimodular}
The goal of this subsection is to show by example that Theorem \ref{ThmUnimodular} does not extend to weak approximate lattices. This shows in particular, that not every weak approximate lattice is an approximate lattice.

Define an action of $\R$ on $\R$ by $\alpha: \R \to {\rm Aut}(\R)$, $\alpha(a).b := e^a b$ and let $G = \bR \rtimes_\alpha \bR$ denote the corresponding semi-direct product. With the Euclidean topology on $\R$ the group G is a compactly generated, non-discrete, amenable and non-unimodular lcsc group. A left-Haar measure on $G$ is given by $dm_G(b,a) = \frac{db da}{e^a}$ and the modular function is $\Delta_G(b,a) = e^{-a}$. 

Since $G$ is non-unimodular, it does not contain any strong approximate lattices by Theorem \ref{ThmUnimodular}. Since $G$ is amenable, every approximate lattice in $G$ is automatically strong, so $G$ does not contain any approximate lattices at all. We will now show that $G$ nevertheless contains a weak approximate lattice.

Indeed, let $\Lambda$ denote the discrete subgroup $\{0\} \rtimes_\alpha \bZ$ of $G$, and note that this subgroup has infinite covolume in $G$. We shall nevertheless show that for a large class of 
admissible probability measures on $G$, there are always stationary probability measures on $X_\Lambda$,
and thus $\Lambda$ is a weak approximate lattice in $G$. It will suffice to construct stationary probability measures on the (non-compact) homogeneous space $G/\Lambda$, since these push-forward to stationary probability measures on the hull $X_\Lambda$ via the canonical map $G/\Lambda \to X_\Lambda$. We note that $G/\Lambda$ can be $G$-equivariantly identified with the direct product $Y := \bR \times \bR/\bZ$ via the map $(b,a) +\Lambda \mapsto (b,a+\bZ)$, and we will work in the latter model.

Now let $m_{\bT}$ denote the Haar probability measure on $\bT = \bR/\bZ$ and define a probability measure $\nu_o = \delta_o \otimes m_{\bT}$ on $Y$. We shall show that if
$\mu$ is any compactly supported admissible probability measure on $G$ satisfying the contraction condition
\begin{equation}
\label{contractive}
\int_G e^{a} \, d\mu(b,a) = \int_G \Delta(g)^{-1} \, d\mu(g)  < 1,
\end{equation}
then  $\mu^{*n} * \nu_o$ converges in the vague topology. In particular, the limit measure then defines a $\mu$-stationary probability measure on $Y$, hence $X_\Lambda$ admits a $\mu$-stationary probability measure for every contractive $\mu$.
\begin{remark} Note that the measures satisfying \eqref{contractive} are very different from the measure used in the proof of Theorem \ref{ThmUnimodular}, which were assumed to satisfy the opposite condition
\[
\int_G \Delta_G(g) \, d\mu(g) < 1.
\]
For this reason, the existence of a $\mu$-stationary probability measure for every contractive $\mu$ does not contradict unimodularity.
\end{remark}

To prove convergence of the measures $\mu^{*n} * \nu_o$ for a contractive $\mu$ we argue as follows. If we abbreviate $[(b,a)] := (b,a+\bZ) \in Y$, then for all $(b_1,a_1),\ldots,(b_n,a_n) \in G$ we have
\[
(b_1,a_1) \cdots (b_n,a_n) \cdot [(b,a)] = \left[\left( \sum_{k=1}^{n-1} e^{A_k} b_k + A_n b, A_n + a\right) \right],
\]
where $A_1 = 0$, and $A_k = a_1 + \ldots + a_{k-1}$ for $k \geq 2$. Hence,
\[
(\mu^{*n} * \nu_o)(f) = \int_{G^n} f\left(\sum_{k=1}^{n-1} e^{A_k} b_k,a\right) \, d\mu(b_1,a_1) \cdots d\mu(b_n,a_n) \, dm_{\bT}(a),
\]
for every $f \in C_o(G)$. In order to show that $(\mu^{*n} * \nu_o)(f)$ converges for every $f$, it thus suffices to 
check that the series 
\[
B_\infty := \sum_{k=1}^{\infty} e^{A_k} b_k
\]
converges for $\mu^{\bN}$-almost every $((b_1,a_1),(b_2,a_2),\ldots) \in G^{\bN}$. Since $\mu$ is compactly supported, there exists some $R > 0$ such that $|b| \leq R$ for $\mu$-almost every $(b,a) \in \supp(\mu)$. 
Hence, by the monotone convergence theorem, the series $B_\infty$ is absolutely convergent $\mu^{\bN}$-almost everywhere if
\[
\sup_n \int_{G^n} \sum_{k=1}^{n-1} e^{A_k} \, d\mu(b_1,a_1) \cdots d\mu(b_n,a_n) < \infty. 
\] 
Since 
\[
\int_{G^n} \sum_{k=1}^{n-1} e^{A_k} \, d\mu(b_1,a_1) \cdots d\mu(b_n,a_n)
= \sum_{k=1}^{n-1} \Big( \int_G e^a \, d\mu(b,a) \Big)^{k-1},
\]
this follows from our assumption \eqref{contractive}. This finishes the proof, and we conclude:
\begin{proposition} 
\begin{enumerate}[(i)]
\item There exist weak approximate lattices which are not approximate lattices. 
\item There exist non-unimodular lcsc groups which admit a weak approximate lattice.
\item There exists subgroups of infinite covolume in lcsc groups which are weak approximate lattices.
\item All three phenomena occur even in amenable lcsc groups.\qed
\end{enumerate}
\end{proposition}
In view of the proposition we believe that the notion of a weak approximate lattice is too weak to allow for a far reaching theory.

\subsection{A unimodular group without approximate lattices}\label{SecCornulier}
Non-unimodularity is really only the first obstruction for a lcsc group to contain a lattice, and there are many unimodular groups which do not contain any lattices. We expect that the same is true for approximate lattices. The following concrete class of examples of unimodular groups without (weak) approximate lattices was pointed out to us by Yves de Cornulier.

\begin{example}[Y. de Cornulier]\label{ExCornulier}
Let $K$ be a compact abelian group and $f: \Z \to K$ be an injective homomorphism with dense image. Denote by $G_f$ the group whose underlying set is given by $\Z \times \Z \times K$ with multiplication given by
\[
(n, m, k) \cdot (n', m', k') := (n+n', m+m', k+k'+f(nm')).
\]
Then $K$ is a maximal compact normal subgroup of $G_f$ which coincides with the center of $G_f$, and $G_f/K \cong \Z^2$, in particular $G_f$ is $2$-step nilpotent and thus unimodular. If we set $Z := f(\Z)$, then $\Z \times \Z \times Z$ is a dense subgroup of $G_f$, which is isomorphic to the integral Heisenberg group.

We claim that $G_f$ does not contain any weak approximate lattices. By Theorem \ref{NilpotentUniform} it suffices to show that it does not contain a uniform approximate lattice. Assume for contradiction that $\Lambda \subset G_f$ was a uniform approximate lattice and denote by $\pi: G_f \to G_f/K \cong \Z^2$ the canonical projection. Since $\Lambda$ is left-syndetic in $G_f$ we deduce that $\pi(\Lambda)$ is left-syndetic in $\Z^2$, and in particular that the set
\[
\{nm' - n'm \mid (n,m), (n', m') \in \pi(\Lambda)\}
\]
is infinite. Since
\[
[(n, m, k), (n', m', k')] = (0, 0, f(nm' - n'm)),
\]
this implies that $[\Lambda, \Lambda] \subset \Lambda^4 \cap K$ is infinite, contradicting the fact that $\Lambda^4$ is locally finite.
\end{example}

\subsection{Amenability} We end this paper by an application which combines measurable and QI techniques in the form of Theorem \ref{ThmUnimodular} and Theorem \ref{MilnorSchwarz}. Recall that a subset of a  proper discrete metric space $(\Lambda, d)$ is called an $(R, \epsilon)$-F\o lner set provided
\[
\frac{|\{x \in \Lambda\mid \max\{d(x, F), d(x, \Lambda \setminus F)\} < R\}|}{|F|} < \epsilon,
\] 
and that $(\Lambda, d)$ is called \emph{metrically amenable} if it admits an $(R, \epsilon)$-F\o lner set for all
for all $R>0$ and $\epsilon>0$. A general (possibly non-discrete) metric space if called \emph{metrically amenable} if it admits a metrically amenable Delone subset. By \cite[Prop. 3.D.35]{CdlH}, metric amenability is invariant under quasi-isometries. In particular, a metric space is metrically amenable if and only if it admits a Delone set and all of its Delone sets are metrically amenable.
\begin{proposition}\label{PropAmenable} Let $G$ be a lcsc group and $\Lambda \subset G$ a finitely generated uniform approximate lattice. Then the following are equivalent:
\begin{enumerate}[(i)]
\item $G$ is amenable (as a lcsc group).
\item $\Lambda$ is metrically amenable (with respect to any metric in its canonical QI class).
\end{enumerate}
\end{proposition}
\begin{proof} From Theorem \ref{MilnorSchwarz} and quasi-isometric invariance of metric amenability we deduce immediately that $\Lambda$ is metrically amenable if and only if $G$ is metrically amenable with respect to any word metric of a compact generating set. It thus suffices to show that $G$ is metrically amenable if and only if it is amenable as a lcsc group. This equivalence is established in \cite[Lemma 4.F.4(2) and Prop. 4.F.8]{CdlH} under the additional assumption that $G$ is unimodular. However, we know from Theorem \ref{ThmUnimodular} that unimodularity of $G$ holds automatically, and the proposition follows.
\end{proof}


\begin{thebibliography}{99}

\bibitem{IRS1}
M. Abert, N. Bergeron, I. Biringer, T. Gelander, N. Nikolov, J. Raimbault, and I. Samet.
\emph{On the growth of $L^2$-invariants for sequences of lattices in Lie groups.}
Preprint, \texttt{http://arxiv.org/abs/1210.2961v3}

\bibitem{IRS2}
M. Abert and Y. Glasner and B. Virag. 
\emph{Kesten's theorem for invariant random subgroups.}
Duke Math. J., 163(3) (2014), 465--488.

\bibitem{BaakeGrimm}
M.~Baake and U.~Grimm.
\emph{Aperiodic order. {V}ol. 1}, Volume 149 of \emph{Encyclopedia of
  Mathematics and its Applications}.
Cambridge University Press, Cambridge, 2013.

\bibitem{BFS}
U. Bader and A. Furman and R. Sauer,
\emph{On the structure and arithmeticity of lattice envelopes.}
C. R. Acad. Sci. Paris, Ser. I 353 (2015), 409--413.

\bibitem{Bjo}
M. Bj\"orklund,
\emph{Product set phenomena for measured groups.}
Ergodic Theory Dyn. Systems, to appear.

\bibitem{BF}
M. Bj\"orklund, A. Fish,
\emph{Ergodic Kneser-type Theorems for amenable groups.}
Preprint, \url{https://arxiv.org/abs/1607.02575}.

\bibitem{BHP}
M. Bj\"orklund and T. Hartnick and F. Pogorzelski,
\emph{Aperiodic order and spherical diffraction.}
Preprint, \url{https://arxiv.org/abs/1602.08928}.

\bibitem{Borel60}
A. Borel,
\emph{Density Properties for Certain Subgroups of Semi-Simple Groups Without Compact
Components.}
Ann. of Math. (2) 72(1), 1960, 179--188.

\bibitem{BoHC}
A. Borel and Harish-Chandra, 
\emph{Arithmetic subgroups of algebraic groups.} 
Ann. of Math. (2) 75 1962 485--535. 

\bibitem{BouGam}
J. Bourgain and A. Gamburd, \emph{Uniform expansion bounds for Cayley graphs of $\SL_2(\mathbb{F}_p)$}, Ann. of Math. (2),
167:2 (2008), 625--642.

\bibitem{Breuillard}
E. Breuillard,
\emph{Lectures on approximate groups},
Lecture Notes, \url{https://www.math.u-psud.fr/~breuilla/ClermontLectures.pdf}.

\bibitem{BrSur}
E. Breuillard, 
\emph{A brief introduction to approximate groups}, Thin groups and superstrong approximation, 
Math. Sci. Res. Inst. Publ., vol. 61, Cambridge Univ. Press, Cambridge, 2014, pp. 23--50.


\bibitem{BrGrTao2}
E. Breuillard, B. Green and T. Tao,  
\emph{The structure of approximate groups.} Publ. Math. IHES 116 (2012), 115--221.


\bibitem{BridsonHaefliger}
M. Bridson and A. Haefliger,
\emph{Metric Spaces of Non-Positive Curvature.}
Grundlehren der mathematischen Wissenschaften, Springer, 1999.

\bibitem{CapraceMonod}
P.-E. Caprace and N. Monod,
\emph{Isometry groups of non-positively curved spaces: discrete subgroups.}
J. Topology 2 No. 4 (2009), 701--746.

\bibitem{CdlH}
Y. Cornulier and P. de la Harpe,
\emph{Metric Geometry of Locally Compact Groups.}
EMS Tracts in Mathematics Vol. 25. 


\bibitem{CorwinGreenleaf}
L. J. Corwin and F. P. Greenleaf,
\emph{Representations of nilpotent Lie groups and their applications.}
Cambridge University Press, 1989.

\bibitem{Dymarz}
T. Dymarz,
\emph{Envelopes of certain solvable groups.}
Comment. Math. Helv. 90 (2015), 195--224.

\bibitem{ElekTardos}
G. Elek and G. Tardos,
\emph{On Roughly Transitive Amenable Graphs and Harmonic Dirichlet Functions.}
Proc. AMS 128(8), 2000, 2479--2485.

\bibitem{ErdSze}
P. Erd\"os and E. Szemer\'edi,
\emph{On sums and products of integers} 
Studies in Pure Mathematics. To the memory of Paul Tur\'an, Basel: Birkh\"auser Verlag, 1983, 213--218.

\bibitem{EskinFarb}
A. Eskin and B. Farb,
\emph{Quasi-flats and rigidity in symmetric spaces.}
Journal of the AMS 10(1997), 653--692.


\bibitem{Freiman}
G. Freiman,
\emph{Foundations of a structural theory of set addition} 
(Translated from the Russian), Translations of Mathematical Monographs, 
Vol. 37, American Mathematical Society, Providence, R. I., 1973, vii+108 pp.

\bibitem{Furman}
A. Furman
\emph{Mostow-Margulis rigidity with locally compact targets.}
Geom. Funct. Anal. 11 (2001), no. 1, 30--59.
\bibitem{Folner}
E. F\o lner,yin
\emph{Note on a generalization of a theorem of Bogoliouboff.} 
Math. Scand. 2, (1954). 224--226. 

\bibitem{URS}
E. Glasner and B. Weiss. 
\emph{Uniformly recurrent subgroups.} 
In: Recent trends in ergodic theory and dynamical systems, Volume 631 of Contemp. Math.,
Amer. Math. Soc., Providence, RI, 2015, 63--75.

\bibitem{Glasner}
S. Glasner, 
\emph{Proximal flows.} 
Lecture Notes in Mathematics, Vol. 517. Springer--Verlag, Berlin-New York, 1976. viii+153 pp.

\bibitem{Gong}
M.-P. Gong,
\emph{Classification of nilpotent Lie algebras of dimension 7 (over algebraically closed fields and $\R$)}. 
Ph.D. thesis, University of Waterloo (Canada), 1998.

\bibitem{Gromov}
M. Gromov,
\emph{Asymptotic invariants of infinite groups.} 
In: Geometric group theory, Volume 2, London Math. Soc. Lecture Note Ser. 182, Cambridge Univ. Press, 1993.

\bibitem{Helfgott}
H. Helfgott, \emph{Growth and generation in $\SL_2(\bZ/p\bZ)$}, Ann. of Math. (2)/ 167:2
(2008), 601--623.

\bibitem{Kaimanovich/Vershik}
V. Kaimanovich and A. Vershik,
\emph{Random walks on discrete groups: boundary and entropy.}
Ann. Probab. 11 (1983), no. 3, 457--490. 

\bibitem{KleinerLeeb}
B. Kleiner and  B. Leeb
\emph{Rigidity of quasi-isometries for symmetric spaces and Euclidean buildings.}
Inst. Hautes Études Sci. Publ. Math. No. 86 (1997), 115--197.

\bibitem{Meyer}
Y. Meyer, 
\emph{Algebraic numbers and harmonic analysis.} 
North-Holland Mathematical Library, Vol. 2. North-Holland Publishing Co., 
Amsterdam-London; American Elsevier Publishing Co., Inc., New York, 1972. x+274 pp. 

\bibitem{Moody}
R. V. Moody,
\emph{Meyer sets and their duals.}
In: \emph{The mathematics of long-range aperiodic order (Waterloo, ON, 1995)},
NATO Adv. Sci.Inst. Ser. C Math. Phys. Sci., 489, Kluwer Acad. Publ., Dordrecht, 1997, 403--441.

\bibitem{Quasiaction}
L. Mosher, M. Sageev and K. Whyte,
\emph{Quasi-actions on trees I: Bounded valence},
Ann. Math. 158 (2003), 115--164.

\bibitem{Pansu}
P. Pansu. 
\emph{Metriques de Carnot-Caratheodory et quasi-isometries des espaces
symetriques de rang un.}
Ann. of Math. (2), 129(1)(1989), 1--60.

\bibitem{Paulin}
F. Paulin,
\emph{De la geometrie et de la dynamique de ${\rm SL}_n(\R)$ et ${\rm SL}_n(\Z)$},
In: N. Berline, A. Plagne and C. Sabbah (ed.), Sur la dynamique des groupes de matrices et applications arithmetiques, 
Editions de l'Ecole Polytechnique (2007), 47--110.

\bibitem{Pl}
H. Pl\"unnecke,
\emph{Eigenschaften und Abschätzungen von Wirkungsfunktionen}, 
Gesellschaft für Mathematik und Datenverarbeitung, Bonn, 1969.

\bibitem{Raghunathan}
M. S. Raghunathan,
\emph{Discrete subgroups of Lie groups.}
Springer, 1972.


\bibitem{Raugi}
A. Raugi, 
\emph{A general Choquet--Deny theorem for nilpotent groups.} 
Annales de l'Institut Henri Poincare (B) Probability and Statistics 40(6) 2004,677--683

\bibitem{Rosenblatt}
J. Rosenblatt, 
\emph{Ergodic and mixing random walks on locally compact groups.}
Math. Ann. 257 (1981), no. 1, 31--42. 

\bibitem{Ruzsa}
I. Ruzsa, 
\emph{An analog of Freiman's theorem in groups.} 
Structure theory of set addition. Astérisque No. 258 (1999), xv, 323--326.

\bibitem{Scheunemann}
J. Scheunemann,
\emph{Two-step nilpotent Lie algebras.}
J. Algebra 7(2), 1967, 152--159.

\bibitem{Tao}
T. Tao, 
\emph{Product set estimates for non-commutative groups.}
Combinatorica 28 (2008), no. 5, 547--594. 

\end{thebibliography}
\end{document}